\documentclass[10pt]{amsart}
\usepackage{amsmath}
\usepackage{amsthm}
\usepackage{amssymb}
\usepackage{mathtools}
\usepackage{graphicx}
\usepackage[multiple]{footmisc}
\usepackage{multicol}
\usepackage{enumerate}
\usepackage{tikz}
\usetikzlibrary{arrows}
\usepackage{hyperref}
\usepackage{float}

\newtheorem{theorem}{Theorem}[subsection]

\theoremstyle{plain}
\ifdefined\theorem \else \newtheorem{theorem}{Theorem} \fi
\ifdefined\maintheorem \else \newtheorem{maintheorem}[theorem]{Main Theorem} \fi
\ifdefined\theoremschema \else \newtheorem{theoremschema}[theorem]{Theorem Schema} \fi
\ifdefined\lemma \else \newtheorem{lemma}[theorem]{Lemma} \fi
\ifdefined\deflemma \else \newtheorem{deflemma}[theorem]{Def-Lemma} \fi
\ifdefined\lemmaschema \else \newtheorem{lemmaschema}[theorem]{Lemma Schema} \fi
\ifdefined\proposition \else \newtheorem{proposition}[theorem]{Proposition} \fi
\ifdefined\corollary \else \newtheorem{corollary}[theorem]{Corollary} \fi
\ifdefined\fact \else \newtheorem{fact}[theorem]{Fact} \fi
\ifdefined\problem \else \newtheorem{problem}[theorem]{Problem} \fi
\ifdefined\conjecture \else \newtheorem{conjecture}[theorem]{Conjecture} \fi
\ifdefined\question \else \newtheorem{question}[theorem]{Question} \fi
\ifdefined\observation \else \newtheorem{observation}[theorem]{Observation} \fi
\newtheorem*{theorem*}{Theorem}
\newtheorem*{lemma*}{Lemma}
\newtheorem*{theoremschema*}{Theorem Schema}
\newtheorem*{proposition*}{Proposition}
\newtheorem*{conjecture*}{Conjecture}
\newtheorem*{question*}{Question}
\newtheorem*{definition*}{Definition}

\theoremstyle{definition}
\ifdefined\definition \else \newtheorem{definition}[theorem]{Definition} \fi
\ifdefined\subdefinition \else \newtheorem{subdefinition}{Definition}[theorem] \fi
\ifdefined\maindefinition \else \newtheorem{maindefinition}[theorem]{Main Definition} \fi
\theoremstyle{remark}
\ifdefined\remark \else \newtheorem{remark}[theorem]{Remark} \fi
\ifdefined\remarks \else \newtheorem{remarks}[theorem]{Remarks} \fi
\ifdefined\example \else \newtheorem{example}[theorem]{Example} \fi
\ifdefined\claim \else \newtheorem{claim}[theorem]{Claim} \fi
\ifdefined\exercise \else \newtheorem{exercise}[theorem]{Exercise} \fi
\newtheorem*{remark*}{Remark}
\ifdefined\acknowledgment \else \newtheorem*{acknowledgment}{Acknowledgment} \fi
\ifdefined\dedication \else \newtheorem*{dedication}{Dedicated} \fi
\ifdefined\case \else \newtheorem*{case}{Case} \fi

\newcounter{my_enumerate_counter}

\newcommand\comment[1]{}

\newcommand\Crm{\mathrm{C}}

\newcommand\Jrm{\mathrm{J}}
\newcommand\Lrm{\mathrm{L}}

\newcommand\Trm{\mathrm{T}}
\newcommand\Vrm{\mathrm{V}}

\newcommand\Zsf{\mathsf{Z}}

\newcommand\Acal{\mathcal{A}}

\newcommand\Dcal{\mathcal{D}}

\newcommand\Ical{\mathcal{I}}
\newcommand\Jcal{\mathcal{J}}

\newcommand\Mcal{\mathcal{M}}

\newcommand\Pcal{\mathcal{P}}

\newcommand\Xcal{\mathcal{X}}
\newcommand\Ycal{\mathcal{Y}}

\newcommand\Pbb{\mathbb{P}}

\newcommand{\dom}{\operatorname{dom}}

\ifdefined\alt \else
  \newcommand{\alt}{\operatorname{alt}}
\fi

\newcommand\Ext{\operatorname{Ext}}

\newcommand{\forces}{\Vdash}

\newcommand\axiom{\mathsf}

\newcommand\CH{\axiom{CH}}

\newcommand\KP{\axiom{KP}}

\newcommand\ZF{\axiom{ZF}}
\newcommand\ZFC{\axiom{ZFC}}

\newcommand\ZFCm{\axiom{ZFC}^-}

\newcommand\GB{\axiom{GB}}

\newcommand\ETR{\axiom{ETR}}

\newcommand\PCA{\Pi_1^1\text{-}\axiom{CA}}
\newcommand\PnCA[1]{\Pi_{#1}^1\text{-}\axiom{CA}}
\newcommand\SnCA[1]{\Sigma_{#1}^1\text{-}\axiom{CA}}
\newcommand\SnAC[1]{\Sigma_{#1}^1\text{-}\axiom{AC}}
\newcommand\SnCC[1]{\Sigma_{#1}^1\text{-}\axiom{CC}}

\newcommand\CC{\axiom{CC}}

\newcommand\KM{\axiom{KM}}
\newcommand\KMCC{\axiom{KMCC}}

\newcommand\ZFCmi{\ZFCm_{\mathrm I}}

\newcommand\AC{\axiom{AC}}

\newcommand\PA{\axiom{PA}}

\newcommand\ACA{\axiom{ACA}}
\newcommand\ATR{\axiom{ATR}}

\newcommand\class{\mathrm}

\newcommand\HOD{\class{HOD}}
\newcommand\Ord{\class{Ord}}

\newcommand\Adm{\class{Adm}}

\newcommand\Add{\class{Add}}

\newcommand{\seq}[1]{{\left\langle #1 \right\rangle}}

\renewcommand{\epsilon}{\varepsilon}

\newcommand\mand{\textrm{ and }}

\newcommand\card[1]{\left\lvert #1 \right\rvert}

\newcommand\rest{\upharpoonright}

\newcommand\powerset{\Pcal}

\newcommand\comp{\circ}

\newcommand{\godel}[1]{\ulcorner#1\urcorner}

\newcommand\Def{\operatorname{Def}}
\ifdefined\Form \else
  \newcommand\Form{\axiom{Form}}
\fi

\newcommand{\impl}{\Rightarrow}
\renewcommand{\iff}{\Leftrightarrow}
\renewcommand{\phi}{\varphi}

\newcommand\SOR{\axiom{SOR}}
\newcommand\ZFCmik[1]{\axiom{ZFC}^-_{\mathrm{I},#1}}
\newcommand\Class{\mathrm{Class}}
\newcommand{\MinMod}{\mathrm{MinMod}}

\DeclareMathOperator{\TC}{TC}

\title{Non-tightness in class theory and second-order arithmetic}

\author{Alfredo Roque Freire}
\address[Alfredo Roque Freire]{University of Aveiro\\
Center for Research \& Development in Mathematics and Applications - CIDMA\\
Campus Universitário de Santiago 3810-193 Aveiro\\
Portugal}
\email{alfrfreire@gmail.com}
\author{Kameryn J. Williams}
\address[Kameryn J. Williams]{
Sam Houston State University \\
Department of Mathematics and Statistics \\
Box 2206 \\
Huntsville, TX 77341-2206 \\
USA}
\email{kameryn.j.w@shsu.edu}
\urladdr{http://kamerynjw.net}

\usepackage[backend=bibtex,style=alphabetic,maxbibnames=15,maxcitenames=6,dateabbrev=false]{biblatex}
\thanks{We thank the anonymous referee for their comments. This research was supported by FCT through CIDMA and projects
UIDB/04106/2020 and UIDP/04106/2020.}

\date{\today}

\begin{document}

\maketitle

\begin{abstract}
A theory $T$ is tight if different deductively closed extensions of $T$ (in the same language) cannot be bi-interpretable. Many well-studied foundational theories are tight, including $\PA$ \cite{Visser2006}, $\ZF$, $\Zsf_2$, and $\KM$ \cite{enayat2017variations}. In this article we extend Enayat's investigations to subsystems of these latter two theories. We prove that restricting the Comprehension schema of $\Zsf_2$ and $\KM$ gives non-tight theories. Specifically, we show that $\GB$ and $\ACA_0$ each admit different bi-interpretable extensions, and the same holds for their extensions by adding $\Sigma^1_k$-Comprehension, for $k \ge 1$. These results provide evidence that tightness characterizes $\Zsf_2$ and $\KM$ in a minimal way.


\end{abstract}

\section{Introduction}


It is well known that set theories like $\ZF$ and class theories like $\GB$ or $\KM$ are capable of interpreting many of their extensions as theories. 
For instance, $\ZF$ interprets $\ZFC + \CH$ via the constructible universe, or one may use the Boolean ultrapower construction over the notion of forcing $\Add(\omega, \omega_2)$ to produce an interpretation of $\ZF + \lnot\CH$ in $\ZF$ (see. \cite{HamkinsSeaboldBooleanUltrapowers}). Accordingly, $\ZF + \CH$ and $\ZF + \lnot\CH$ are mutually interpretable.
Mutual interpretability yields equiconsistency results but for many set-theoretical purposes it is a coarse notion of equivalence.
The issue is that we may lose information. For an example where this loss is severe, consider a model of $\ZFC + \neg \CH$ with a measurable cardinal. Carry out the constructible universe interpretation of $\ZF + \CH$ in this model followed by the Boolean ultrapower interpretation. This produces a model of $\ZFC + \neg \CH$ again, but the model cannot have a measurable cardinal. (Because as a forcing extension of $\Lrm$ it will miss, for example, $0^\sharp$.) Even in less severe cases information is lost---the boolean ultrapower interpretation produces ill-founded models, so this two-step interpretation destroys information about well-foundedness.

\begin{figure}[H]\label{informationLoss}
\begin{center}
\tikzset{every picture/.style={line width=0.75pt}} 
\begin{tikzpicture}[x=0.75pt,y=0.75pt,yscale=-1,xscale=1]

	\draw  [fill={rgb, 255:red, 245; green, 166; blue, 35 }  ,fill opacity=0.5 ] (132.97,131.16) .. controls (132.97,96.73) and (143.72,68.81) .. (156.97,68.81) .. controls (170.22,68.81) and (180.96,96.73) .. (180.96,131.16) .. controls (180.96,165.59) and (170.22,193.5) .. (156.97,193.5) .. controls (143.72,193.5) and (132.97,165.59) .. (132.97,131.16) -- cycle ;
	\draw  [fill={rgb, 255:red, 74; green, 144; blue, 226 }  ,fill opacity=0.5 ] (273,131.16) .. controls (273,96.73) and (283.74,68.81) .. (296.99,68.81) .. controls (310.24,68.81) and (320.99,96.73) .. (320.99,131.16) .. controls (320.99,165.59) and (310.24,193.5) .. (296.99,193.5) .. controls (283.74,193.5) and (273,165.59) .. (273,131.16) -- cycle ;
	\draw  [fill={rgb, 255:red, 74; green, 144; blue, 226 }  ,fill opacity=0.7 ] (141.63,162.33) .. controls (141.63,145.11) and (148.49,131.16) .. (156.97,131.16) .. controls (165.44,131.16) and (172.31,145.11) .. (172.31,162.33) .. controls (172.31,179.54) and (165.44,193.5) .. (156.97,193.5) .. controls (148.49,193.5) and (141.63,179.54) .. (141.63,162.33) -- cycle ;
	\draw  [fill={rgb, 255:red, 245; green, 166; blue, 35 }  ,fill opacity=0.7 ] (281.65,110.31) .. controls (281.65,87.39) and (288.52,68.81) .. (296.99,68.81) .. controls (305.47,68.81) and (312.33,87.39) .. (312.33,110.31) .. controls (312.33,133.23) and (305.47,151.81) .. (296.99,151.81) .. controls (288.52,151.81) and (281.65,133.23) .. (281.65,110.31) -- cycle ;
	\draw [fill={rgb, 255:red, 74; green, 144; blue, 226 }  ,fill opacity=1 ] [dash pattern={on 0.84pt off 2.51pt}]  (156.97,131.16) -- (296.99,68.81) ;
	\draw [shift={(231.55,97.95)}, rotate = 156] [fill={rgb, 255:red, 0; green, 0; blue, 0 }  ][line width=0.08]  [draw opacity=0] (10.72,-5.15) -- (0,0) -- (10.72,5.15) -- (7.12,0) -- cycle    ;
	\draw  [dash pattern={on 0.84pt off 2.51pt}]  (156.97,193.5) -- (296.99,193.5) ;
	\draw [shift={(231.98,193.5)}, rotate = 180] [fill={rgb, 255:red, 0; green, 0; blue, 0 }  ][line width=0.08]  [draw opacity=0] (10.72,-5.15) -- (0,0) -- (10.72,5.15) -- (7.12,0) -- cycle    ;
	\draw  [dash pattern={on 0.84pt off 2.51pt}]  (156.97,68.81) -- (296.99,68.81) ;
	\draw [shift={(220.48,68.81)}, rotate = 0] [fill={rgb, 255:red, 0; green, 0; blue, 0 }  ][line width=0.08]  [draw opacity=0] (10.72,-5.15) -- (0,0) -- (10.72,5.15) -- (7.12,0) -- cycle    ;
	\draw  [dash pattern={on 0.84pt off 2.51pt}]  (156.97,193.5) -- (296.99,151.81) ;
	\draw [shift={(220.75,174.51)}, rotate = 343.42] [fill={rgb, 255:red, 0; green, 0; blue, 0 }  ][line width=0.08]  [draw opacity=0] (10.72,-5.15) -- (0,0) -- (10.72,5.15) -- (7.12,0) -- cycle    ;
	\draw    (77.12,154.95) .. controls (97.16,134.91) and (134.59,141.31) .. (151.25,164.06) ;
	\draw [shift={(152.25,165.48)}, rotate = 235.98] [color={rgb, 255:red, 0; green, 0; blue, 0 }  ][line width=0.75]    (10.93,-3.29) .. controls (6.95,-1.4) and (3.31,-0.3) .. (0,0) .. controls (3.31,0.3) and (6.95,1.4) .. (10.93,3.29)   ;
	\draw    (366.61,97.53) .. controls (364.29,74.28) and (318.46,73.17) .. (299.04,102.46) ;
	\draw [shift={(298.17,103.82)}, rotate = 301.61] [color={rgb, 255:red, 0; green, 0; blue, 0 }  ][line width=0.75]    (10.93,-3.29) .. controls (6.95,-1.4) and (3.31,-0.3) .. (0,0) .. controls (3.31,0.3) and (6.95,1.4) .. (10.93,3.29)   ;

	\draw (27.01,157.89) node [anchor=north west][inner sep=0.75pt]  [font=\scriptsize] [align=left] {interpreted models};
	\draw (340.51,100.47) node [anchor=north west][inner sep=0.75pt]  [font=\scriptsize] [align=left] {interpreted models};
	\draw (118.04,45.97) node [anchor=north west][inner sep=0.75pt]  [font=\footnotesize] [align=left] {Models of $\ZF + \CH$};
	\draw (251.69,45.87) node [anchor=north west][inner sep=0.75pt]  [font=\footnotesize] [align=left] {Models of $\displaystyle\ZF + \neg \CH$};
\end{tikzpicture}
\end{center}
\caption{The $\Lrm$ and Boolean ultrapower interpretations}
\end{figure}
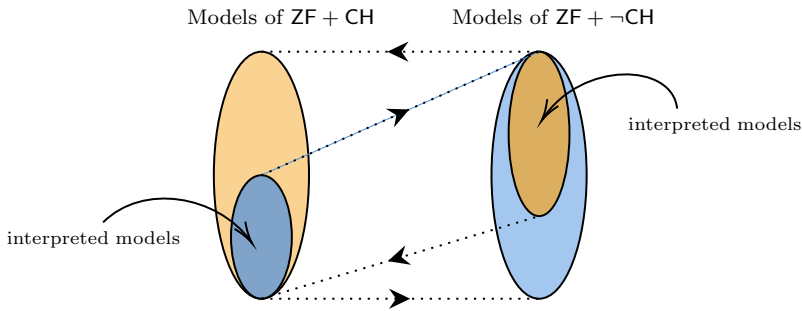

To avoid this limitation and properly establish an equivalence between theories, we need a bi-interpretation.\footnote{
Interpretations were formally introduced by Tarski in \cite{tarski1953} to deal with undecidable theories. Arguably, we can say that interpretations were introduced informally in the study of the consistency of alternative axioms of geometry in the second half of the XIX century (see \cite[p. 260]{hodges1993model}). Feferman studied interpretations themselves as mathematical objects in \cite{feferman1960, feferman1962}. Finally, bi-interpretations appear as faithful interpretations in \cite{feferman19621} even though their definition still do not encompass the full extent of what we now call bi-interpretations.}
Precisely, a \emph{bi-interpretation} between two theories $T_1$ and $T_2$ are interpretations $\Ical$ and $\Jcal$, respectively of $T_1$ in $T_2$ and of $T_2$ in $T_1$, so that there are definable functions $f$ and $g$ in $T_1$ and $T_2$ such that for all $\varphi$ and $\psi$
$$T_1 \vdash \varphi(x_1, x_2, \ldots, x_n) \leftrightarrow \varphi^{\Ical^\Jcal}(f(x_1), f(x_2), \ldots, f(x_n))$$
And 
$$T_2 \vdash \psi(x_1, x_2, \ldots, x_k) \leftrightarrow \psi^{\Jcal^\Ical}(g(x_1), g(x_2), \ldots, g(x_k)).$$
One can also view bi-interpretations on the level of individual models, saying that models $M_1$ and $M_2$ are bi-interpretable. One should view the theory definition as a uniform version of the model definition, saying that every model of the theory $T_1$ is bi-interpretable with a model of the theory $T_2$, and vice versa, always using the same choice of interpretations.



Notably, theories such as PA and ZF cannot fix one model up to isomorphism, i.e. they are not categorical. We can however say that the second order versions of PA and ZF are respectively categorical (Dedekind in \cite{Dedekind1965}) and quasi-categorical (Zermelo in \cite{zermelo1930grenzzahlen}). 
These categoricity results assume that second order quantifiers indeed range over all subsets of a given domain. Thus it requires us to attribute to the mathematician a not substantiated ability to fix the meaning of second order concepts.\footnote{See \cite{button2016structure, putnam1980models} for a detailed philosophical analysis of the categoricity results.}.
In this context, the semantic notion of bi-interpretation can be understood as a weak form of `sameness' allowing for a weaker form of categoricity that do not rely on an arbitrary reference to the fullness of second order quantifiers.\footnote{Alternatively to considering bi-interpretations, one may consider a single model with two versions of the same theory (e.g. two separate symbols for $\in$). Proving that the two versions are always isomorphic amounts to what is call Internal Categoricity. This concept is nonetheless limited in scope as it uses the axiom-schemes of the theory to allow for the two models to be related (e.g. separation satisfied by each model should include formulas with the alternative symbol for membership). This concept was introduced by Parson \cite{parsons1990structuralist} with respect to arithmetic. Väänänen and Wang in \cite{vaananen2012second, vaananen2015internal, vaananen2020tracing} studied the internal categoricity in set theory and further advanced the topic in recent years.}  
Instead of asserting that two models have isomorphic ontologies, a bi-interpretation equates the expressible ontology of possibly different models.

In pursuing this form of categority for set and class theories, one should examine which theories do not admit bi-interpretable models.
Indeed, the interpretations of $\ZF + \CH$ and $\ZF + \neg\CH$ given above lose information. But, are there different interpretations that do not have this problem, and instead give a bi-interpretation? 
As demonstrated by Enayat \cite{enayat2017variations}, any two bi-interpreted models of $\ZF$ are isomorphic. Consequently, no two different extensions of $\ZF$ are bi-interpretable, and so the answer is negative. 
While a theory like $\ZF$ has many models due to the incompleteness phenomenon, in a sense we cannot have ``too many''. This property of $\ZF$ was first investigated by Visser in \cite{Visser2006} with respect to arithmetic and later named by Enayat as tightness.\footnote{Philosophically, as bi-interpretation deals with expressible ontology, the fact that different models of ZF are never bi-interpretable suggests that universalist set theorists `living' in different universes can only assert to fully understand each other if both believe the other is wrong about their own intuitions. Not only the other is wrong about the statement that their model is the model of set theory, but also that their own intuitions about the structure of their models are wrong. A detailed analysis of this dynamic can be found in \cite{unreducibleFreire}.}

\begin{definition*}
A theory $T$ is \emph{tight} if every two bi-interpretable extensions of $T$ in the same language as $T$ have the same deductive closure.
\end{definition*}

\begin{definition*}
A theory $T$ is \emph{semantically tight} if every two bi-interpretable models of $T$ are isomorphic.\footnote{
Enayat works with a stronger notion he calls solid.
Consider $N$ is a $M$-definable model and that $N$ has a definable copy $\overline{M}$ of $M$. Using the same interpretation, $\overline{M}$ obtains a definable copy $\overline{N}$ of $N$. Saying that $M$ and $N$ are bi-interpretable amounts to (i) there is a $M$-definable isomorphism from $M$ to $\overline{M}$ and (ii) there is a $N$-definable isomorphism from $N$ to $\overline{N}$. If we can prove that models $M$ and $N$ of a theory $T$ are isomorphic without assuming (ii), we can say that $T$ is not only semantically tight but also solid.}
\end{definition*}

It is evident that every semantically tight theory is also tight.

\begin{theorem*}[Enayat]
$\ZF$, $\KM$, $\Zsf_2$ are tight and semantically tight.\footnote{Enayat proved that $\ZF$, $\KM$, $\Zsf_2$ are tight \cite{enayat2017variations}; Visser proved that $\PA$ is tight \cite{Visser2006}. Visser and Friedman also proved the $\ZF$ case in an unpublished work (see note 1 of \cite{enayat2017variations}). Indeed, Enayat shows the stronger result that all these theories are solid. Note, however, that the results in this paper concern non-tightness and hence will trivially imply non-solidity.
For a brief exposition on this, we recommend Hamkins's blog post \cite{biinterpretationPost}.}
\end{theorem*}

A natural followup question is whether there are tight subsystems of these theories. This was proposed in \cite[p. 14]{enayat2017variations} and partially addressed with respect to $\ZF$ by Freire and Hamkins in \cite{bwst2020}. They show that there are bi-interpretable models of $\Zsf$ and $\ZFCm$.\footnote{$\Zsf$ refers to the first order version of Zermelo set theory composed of $\ZF$ without the Replacement schema; $\ZFCm$ stands for $\ZFC$ without the Powerset axiom. Note that $\ZFCm$ should be axiomatized with Collection schema, not Replacement, and the well-ordering theorem instead of Zermelo's formulation of choice, as these are not equivalent in the absence of Powerset \cite{zarach1996}.} 
Moreover, since these model constructions were uniformly produced, they obtained different bi-interpretable extensions of $\Zsf$ and $\ZFCm$. A full answer to this question amounts to a profound characterization of $\ZF$ and it should be done by obtaining bi-interpretable models of the theory $\Zsf$ with fragments of the axiom scheme of replacement.

In a similar light, this article investigates tightness for subsystems of $\KM$, obtained by restricting the Comprehension axiom. The weakest subsytem in this hierarchy is G\"odel--Bernays class theory $\GB$, where Comprehension is only allowed for first-order formulae. Strengthening upward adding Comprehension for $\Sigma^1_k$ formulae gives theories we will call $\KM_k$.

Our first main results are that $\GB$ and $\KM_k$ are not semantically tight.

\begin{maintheorem}
Let $\kappa$ be an inaccessible cardinal and suppose $\Vrm_\kappa \models \Vrm = \HOD$. 
\begin{itemize}
\item The minimum model of $\GB$ over $\Vrm_\kappa$ is bi-interpretable with a certain extension adding a Cohen-generic class of ordinals. Thus, $\GB$ is not semantically tight.
\item Let $k \ge 1$. The minimum model of $\KM_k$ over $\Vrm_\kappa$ is bi-interpretable with a certain extension adding a Cohen-generic class of ordinals. Thus, $\KM_k$ is not semantically tight.
\end{itemize}
\end{maintheorem}

We then build on these to show that, indeed, these theories are not tight, and that the same is true for subsystems of $\Zsf_2$. The failure of tightness should be seen as a uniform version of the failure of semantic tightness.

\begin{maintheorem}
The following theories are not tight.
\begin{itemize}
    \item $\GB$;
    \item $\KM_k$ for $k \ge 1$; and
    \item $\KM_k + \Sigma^1_k$-Class Collection, for $k \ge 1$; and
    \item Any of the above theories plus the schema of Replacement for all second-order formulae.
\end{itemize}
\end{maintheorem}

While our primary interest is in class theories, our methods are flexible enough to also apply to subsystems of second-order arithmetic.

\begin{maintheorem}
The following theories are not tight.
\begin{itemize}
    \item $\ACA_0$;
    \item $\PnCA{k}_0$ for $k \ge 1$;
    \item $\SnAC{k}_0$ for $k \ge 1$; and
    \item Any of the above theories plus the full Induction schema, i.e. the theories $\ACA$, $\PnCA{k}$, and $\SnAC{k}$.
\end{itemize}
\end{maintheorem}

In forthcoming work, Ali Enayat \cite{enayat:in-progress} independently investigated the nontightness of fragments of $\KM$ and $\Zsf_2$. He showed that finitely axiomatizable subtheories of these are not tight. That gives an alternate proof of the nontightness of $\GB$, $\KM_k$, $\ACA_0$, and $\PnCA{k}_0$, as well as the versions with Class Collection or $\AC$. But the second-order Replacement schema and the full Induction schema are not finitely axiomatizable, so his methods don't apply to the theories with those schemata.

We present the semantic non-tightness (Section~\ref{semanticNonTightness}) and non-tightness (Section~\ref{NonTightness}) results for class theory separately. The constructions for non-tightness amounts to more difficult variants of the constructions for semantic non-tightness, generalized to apply to a wider class of models, including nonstandard models, so we present the easier constructions first. Additionally, the constructions over $\Vrm_\kappa$ may be of interest to the set theorist with no interest in nonstandard models, and we wish to be accommodating to any such reader. 
In Section~\ref{NonTightnessArithmetic}, we explore how to apply the same technique to subsystems of second-order arithmetic. And in Section~\ref{sec:final} we briefly discuss the extent to which our constructions generalize and what questions remain open.

Before these sections we recall some definitions and basic facts about class theory.

\section{Review of class theories and class forcing} \label{sec:class-theories}

In this paper we look at class theories, also called second-order set theories, those set theories that have proper classes as objects in their domains of discourse. We will use a two-sorted approach, writing a model as e.g. $(M,\Xcal)$ with $M$ being the sets of the model and $\Xcal$ being the classes of the model. Following standard convention, when writing formulae in the language of class theory we will use lowercase variables for sets and uppercase variables for classes. For example, $\forall x \exists Y \forall z\ (z \in x \iff z \in Y)$ asserts that every set is co-extensional with some class, a trivial consequence of First-Order Comprehension.

If a formula only quantifies over sets---but possibly has class parameters---we call it \emph{first-order}. The class of first-order formulae is denoted with any of $\Sigma^1_0$, $\Pi^1_0$, or $\Delta^1_0$. From the first-order formulae we build up the hierarchy of $\Sigma^1_k$ and $\Pi^1_k$ formulae by adding class quantifiers in the way familiar to any logician. Namely, a $\Sigma^1_k$ formula is of the form $\exists \bar X_1 \cdots \forall \bar X_k \phi(\bar X_1, \ldots \bar X_k)$, where $\phi$ is first-order and there are $k$ many blocks of alternating class quantifiers, while a $\Pi^1_k$ formula is of the form $\forall \bar X_1 \cdots \exists \bar X_k \phi(\bar X_1, \ldots \bar X_k)$, again with first-order $\phi$ and $k$ many blocks of alternating class quantifiers.


\begin{definition}
\emph{G\"odel--Bernays} class theory $\GB$ is axiomatized by the following.
\begin{itemize}
\item $\ZFC$ for classes;\footnote{The models we consider in this paper will all satisfy $\Vrm = \HOD$ and hence satisfy the axiom of choice. So for our purposes we do not want to use merely $\ZF$ for the sets. For this same reason our models will for free satisfy Global Choice.}
\item Class Extensionality, asserting that two classes are equal if and only if they have the same elements;
\item Class Replacement, asserting that the image of a set under a class function is always a set; and
\item First-Order Comprehension, asserting that classes can be defined using comprehension for first-order formulae. More precisely, this axiom schema has as instances the universal closure of
\[
\exists X\ X = \{ y : \phi(y,\bar P) \}
\]
for each first-order formula $\phi$.
\end{itemize}
If we add to $\GB$ Full Comprehension, viz. the instances of Comprehension for any formula in the language of class theory, we get \emph{Kelley--Morse} class theory $\KM$. For finite $k \ge 1$, adding $\Sigma^1_k$-Comprehension---Comprehension for $\Sigma^1_k$ formulae---gives the theory we will call $\KM_k$.
\end{definition}

For some purposes $\KM$ is insufficient, and needs to be extended by a version of Collection for classes.\footnote{For example, $\KM$ does not suffice to prove a class version of Fodor's lemma \cite{gitman-hamkins-karagila:2021}, but adding Class Collection enables the proof of class Fodor's lemma.}
It will be convenient in this paper to work with this stronger variant (but not stronger in consistency strength). Using a stronger version will not harm our results since tightness is preserved by extension in the same language.

A \emph{hyperclass} is a collection of classes. Formally these do not exist in our models similar to how classes formally do not exist as objects in $\ZFC$. However, some hyperclasses we can code with an individual class.

\begin{definition}\label{hyperclassCode}
A \emph{code} for a hyperclass is a class of ordered pairs. We say that a hyperclass $\Acal$ is coded if there is a code $A$ so that 
\[
\Acal = \big\{ (A)_i : i \in \Vrm \big\},
\]
where $(A)_i = \{ x : (i,x) \in A \}$ is the \emph{$i$-th slice} of $A$.
\end{definition}

\begin{definition} \label{def:cc}
The \emph{Class Collection} ($\CC$) axiom schema asserts that if for each set there is a class satisfying some property, then we can collect witnesses classes into a single coded hyperclass.\footnote{In the context of second-order arithmetic, the analogous axiom schema is referred to as $\AC$; cf. Definition~\ref{def:soas}.} 
Formally, instances of this schema are the universal closure of
\[
\forall x \exists Y\ \phi(x,Y,\bar P) \impl \exists B\ \forall x \exists i\ \phi(x,(B)_i,\bar P),
\]
where $\phi$ ranges across all formulae in the language of class theory. If we restrict this schema to $\Sigma^1_k$ formulae we get \emph{$\Sigma^1_k$-Class Collection} ($\SnCC{k}$).
\end{definition}

Marek and Mostowski \cite[Theorem 2.5]{marek-mostowski1975} showed that given any model of $\KM$ you can thin down the classes to get a model of $\KMCC = \KM + \CC$ with the same sets. Ratajczyk \cite{ratajczyk1979} built on their work to show that given any model of $\KM_k$ you can thin down the classes to get a model of $\KMCC_k = \KM_k + \SnCC{k}$ with the same sets, where $k > 0$.

Just as Collection yields that every formula in the language of set theory is equivalent to one in the L\'evy hierarchy, $\CC$ yields that every formula in the language of class theory is equivalent to a $\Sigma^1_k$ formula for some $k$. So the theories $\KMCC_k$ give a hierarchy of stronger and stronger theories which in the union give the full theory $\KMCC$.

\subsection{Bi-interpretability with first-order set theory}\label{bi-interpretability-first-order}

For some of our results it will be convenient to work with first-order set theories rather than class theories. The construction behind the more difficult direction of these bi-interpretations goes back to Scott \cite{scott1960}. The key observation is that the Foundation axiom implies that every set $x$ is determined by the isomorphism type of $(\TC(\{x\}),\in)$. As such, sets can be represented with isomorphism classes of well-founded, extensional directed graphs with a maximum element. In this way a model of $\GB$ can represent sets of rank $\mathord{>}\Ord$. To have a name, call this construction the \emph{unrolling construction} and refer to the model of first-order set theory obtained as the \emph{unrolled model}.

\begin{theorem}[\cite{marek-mostowski1975}]
$\KMCC$ and $\ZFCm$ plus ``there is a largest cardinal, and it is inaccessible'' are bi-interpretable.\footnote{See \cite{antos-friedman2017} for a modern treatment of this result.}
\end{theorem}

Denote this latter theory by $\ZFCmi$. Working in $\ZFCmi$ let $\kappa$ denote the largest cardinal. For these bi-interpretability results, $\Vrm$ of the model of class theory is isomorphic to $\Vrm_\kappa$ of the unrolled model of $\ZFCmi$. That is, the sets are fixed and the bi-interpretation is entirely about what happens in the classes.

Doing the construction more carefully you can get versions of this result for restricted amounts of Comprehension.
Here, let $\ZFCmik{k}$ denote the theory obtained from $\ZFCmi$ by restricting the Collection and Separation schemata to $\Sigma_k$ formulae.


\begin{theorem}[\cite{ratajczyk1979}]\label{companionModel}
The following pair of theories are bi-interpretable, for $k \ge 1$.
\begin{itemize}
\item $\KMCC_k$ and $\ZFCmik{k}$.
\end{itemize}
\end{theorem}

The reader who desires to read through the construction in detail is referred to the second author's dissertation \cite{williams-diss}.

The main utility of these bi-interpretation results is that they allow us to use known facts about models of first-order set theory to draw conclusions about models of class theory. Additionally, some arguments become easier to formulate in that context, since we have access to von Neumann ordinals, the Mostowski collapse theorem, and so on, whereas with classes we don't have direct access to these powerful tools.

\subsection{Class forcing}

We will use class forcing over models of $\KMCC_k$. Because the theory of this is less well-known than over models of $\KM$ or $\GB$, we recall the important facts here. Let's begin by addressing nonstandard models.

With a transitive model of set theory, given a generic $G$ you can interpret all $\Pbb$-names via an induction external to the model. If a model of set theory is ill-founded, we cannot do that. Instead we need an approach similar to the Boolean ultrapower approach. The atomic forcing relations $p \forces \sigma = \tau$ and $p \forces \sigma \in \tau$ yield the equivalence relation $=_G$ defined as $\sigma =_G \tau$ if and only if $p \forces \sigma = \tau$ for some $p \in G$ and a similarly defined congruence $\in_G$ modulo $=_G$. Quotienting the $\Pbb$-names by $=_G$ and using $\in_G$ as the membership relation gives the forcing extension. Identifying the ground model with the collection of $\check x/\mathord{=_G}$ for check names $\check x$, we get the forcing extension as a genuine extension. It is straightforward to check that in case you start with a transitive model, this produces a model isomorphic to the one obtained by the external induction. The usual lemmata about forcing can be proved in this context.\footnote{For a recent exposition of these details, with a focus on how it makes sense for ill-founded models, we recommend \cite{GHHSW2020}.}

\begin{theorem}[Stanley, S.D. Friedman \cite{stanley1984,friedman:book}] 
$\GB$ proves that pretame class forcings satisfy the forcing theorem for first-order formulae, viz. that the relations $p \forces \phi(\sigma,\ldots)$ are classes for each first-order formula $\phi$.
\end{theorem}

\begin{corollary}
Forcing with a tame class forcing preserves all axioms of $\GB$ or $\KM$.
\end{corollary}

We elide the technicalities of tameness and pretameness, and point the reader to \cite{friedman:book} or \cite{AntosGitman:ModernClassForcing}. What is needed for our purposes is that $\Add(\Ord,1)$, the forcing to add a Cohen-generic class of ordinals, is tame.

\begin{theorem} \label{thm:forcing-preserves-axioms}
Let $k \ge 1$. Forcing with a tame class forcing preserves all axioms of $\KMCC_k$.
\end{theorem}

\begin{proof}[Proof Sketch]
One way to prove this goes through the bi-interpretability with first-order set theory. Knowing that set forcing preserves $\ZFCm_k$, with a little work one concludes tame class forcing preserves $\KMCC_k$.

Alternatively, one can prove this directly within class theory. To prove $\Sigma^1_k$-Comprehension and Class Collection, first you need to prove that the forcing relations for $\Sigma^1_k$-formulae are classes. Note that this uses $\SnCC{k}$ to be able to pull the set quantifiers expressing ``densely many conditions force such and such'' inside class quantifiers, so you get a $\Sigma^1_k$-definition for the forcing relations. (Compare to, in the $\ZF$ context, how replacement is used to pull bounded quantifiers inside to get that the forcing relation for a $\Sigma_k$ formula is $\Sigma_k$.) Once you know these forcing relations are classes, you then prove the preservation of the axioms in the usual way.
\end{proof}

\begin{definition}
Suppose $(M,\Xcal)$ and $(M,\Ycal)$ are two models of class theory with the same sets $M$. Say that $(M,\Ycal)$ is a \emph{width-extension} of $(M,\Xcal)$ if $\Xcal \subseteq \Ycal$ and for every well-order $\Gamma \in \Ycal$ there is $\Gamma' \in \Xcal$ so that $(M,\Ycal)$ has an isomorphism $\Gamma \cong \Gamma'$.
\end{definition}

This notion is a class theoretic cousin of the familiar notion in first-order set theory of an extension which does not contain any new ordinals. As in the $\ZF$ context, forcing gives a width-extension (assuming strong enough axioms in the ground model).

\begin{theorem} \label{thm:width-extensions}
Let $k \ge 1$. Forcing over a model of $\KMCC_k$ with a pretame forcing produces a width extension.
\end{theorem}

\begin{proof}
Hamkins and Woodin \cite{HamkinsWoodin2018} proved that pretame forcing over a model of Open Class Determinancy cannot add new ordertypes for well-orders. Since $\Sigma^1_1$-Comprehension is enough to prove determinancy for open class games \cite{gitman-hamkins2017}, this gives the result.

Alternatively, one can prove this via the bi-interpretability with first-order set theory, using that forcing over a model of $\ZFCm_1$ cannot add new ordinals.
\end{proof}

As an aside, we remark that it is open whether $\KMCC_1$ is necessary for this result.

\begin{question}
Does $\GB$ prove that every pretame forcing extension is a width extension?
\end{question}

\section{Semantic non-tightness in class theory}\label{semanticNonTightness}

In this section we show that certain fragments of $\KM$ fail to be semantically tight. All models considered will have the same sets. Namely, they will be $\Vrm_\kappa$ for a fixed inaccessible cardinal $\kappa$, and we will assume that $\Vrm_\kappa \models \Vrm = \HOD$. (It is easy to arrange such by forcing, if necessary. Alternatively, this can be obtained by restricting down to an inner model.) 

It is well known that satisfying $\Vrm = \HOD$ is equivalent to having a definable (without parameters) global well-order. We will use the slightly stronger fact that there is a uniform definition which works for any model of $\Vrm = \HOD$. Namely, $\Vrm = \HOD$ asserts that every set is definable in some $\Vrm_\beta$ using some ordinal parameter $\alpha$. So if we order the sets $x$ by the least $\beta$, then the least formula $\phi(v_1,v_2)$, then the least parameter $\alpha$ so that $x$ is defined by $\phi$ in $\Vrm_\beta$ using parameter $\alpha$, this gives a global well-order of the universe in ordertype $\Ord$. We will call this the \emph{$\HOD$-order}, refer to \emph{$\HOD$-least} choices, and so on. 

\subsection{Semantic non-tightness of \texorpdfstring{$\GB$}{GB}}

The strategy for establishing the semantic non-tightness of $\GB$ is this. Using $\Vrm_\kappa$ as the sets there is a minimum model of $\GB$, namely $(\Vrm_\kappa, \Def(\Vrm_\kappa))$ where we append the first-order definable subsets of $\Vrm_\kappa$ to be the classes. By Tarski's theorem on the undefinability of truth, being in $\Def(\Vrm_\kappa)$ cannot be first-order definable over $\Vrm_\kappa$. But it is second-order definable and indeed absolutely so. Moreover, we will produce a carefully defined $\Crm \subseteq \kappa$ which is also absolutely second-order definable over $\Vrm_\kappa$. Our two bi-interpretable models of $\GB$ will then be $(\Vrm_\kappa,\Def(\Vrm_\kappa))$ and $(\Vrm_\kappa,\Def(\Vrm_\kappa;\Crm))$ where $\Def(\Vrm_\kappa;\Crm)$ denotes the hyperclass of subsets of $\Vrm_\kappa$ definable using $\Crm$ as a parameter.

\begin{observation}\label{obs:truth-definability}
The satisfaction predicate $\Trm$ for $\Vrm_\kappa$ is both $\Sigma^1_1$ and $\Pi^1_1$ definable over $\Vrm_\kappa$. If $\Xcal \subseteq \powerset(\Vrm_\kappa)$ is any possible collection of classes which give a model of $\GB$ then $\Xcal$ will correctly define $\Trm$.
\end{observation}

The content of this observation can be traced back to Mostowski \cite{mostowski1951}.

\begin{proof}
To define $\Trm$ in a $\Sigma^1_1$ way, we observe that it is the union of the $\Sigma_k$-satisfaction predicates, and these all agree on their common domains. While the first-order definitions of these are progressively more complex as $k$ increases, whether a class is a $\Sigma_k$-satisfaction class is uniformly recognizable in $k$. Namely, $S$ is the $\Sigma_k$-satisfaction class if it satisfies the Tarskian recursion on its domain and it judges the truth of all and only the $\Sigma_k$ formulae. To define $\Trm$ in a $\Pi^1_1$ way, $\phi[\vec a] \in \Trm$ iff for any class $S$ if $S$ is a $\Sigma_k$-satisfaction class and $\phi[\vec a]$ is in its domain, then $S$ judges $\phi[\vec a]$ to be true.

Observe that $\GB$ suffices to prove the $\Sigma_k$-satisfaction classes exist. So these definitions work for any model of $\GB$ with $\Vrm_\kappa$ as its sets. (Here we use that we are working over a transitive model and so there are only standard $k$ to worry about.)
\end{proof}

While $\Def(\Vrm_\kappa)$ is a hyperclass and thus cannot be a class in any model of class theory with $\Vrm_\kappa$ as its sets, it can be coded by a single class.

\begin{observation}
After a minor reshuffling of coordinates, $\Trm$ is a code for $\Def(\Vrm_\kappa)$.
\end{observation}

\begin{proof}
A class $X$ is definable if and only if $X = \{ x : \phi[x,\vec a] \in \Trm \}$ for some formula $\phi$ with parameters $\vec a$. So by reshuffling coordinates in $\Trm$ to consist of ordered pairs $((\phi,\vec a),x)$ we get a code for $\Def(\Vrm_\kappa)$.
\end{proof}

We will slightly abuse notation and use $\Trm$ to refer both to the satisfaction class and to this code for $\Def(\Vrm_\kappa)$. We will write $(\Trm)_\xi$ to refer to the slice of $\Trm$ corresponding to the $\xi$-th pair $(\phi,\vec a)$ in the $\HOD$-order.

If $\Crm$ is a second-order definable generic for a forcing $\Pbb \in \Def(\Vrm_\kappa)$ then similar results hold for $\Trm(\Crm)$, the satisfaction class relative to $\Crm$ as a parameter, and $\Def(\Vrm_\kappa;\Crm)$, the hyperclass of classes definable using $\Crm$ as a parameter.

\begin{lemma} \label{lem:crm-relative-truth}
Suppose $\Crm$ is a generic over $(\Vrm_\kappa,\Def(\Vrm_\kappa))$ for a forcing $\Pbb \in \Def(\Vrm_\kappa)$, and $\Crm$ is second-order definable. Then, $\Trm(\Crm)$ is definable, indeed definable in a uniform manner across all $(\Vrm_\kappa,\Xcal) \models \GB$ which define $\Crm$ the same. Moreover, after a minor reshuffling of coordinates $\Trm(\Crm)$ is a code for $\Def(\Vrm_\kappa;\Crm)$.
\end{lemma}

\begin{proof}
The reason this isn't completely trivial is that partial satisfaction classes relative to $\Crm$ will not be (first-order) definable unless $\Pbb$ is trivial, and so we cannot just relativize the definition of $\Trm$. Instead, we use the forcing theorem: $\phi[\vec a] \in \Trm(\Crm)$ if and only if there is $p \in \Crm$ so that ``$p \forces \phi(\vec a,\check C)\text{''} \in \Trm$. By the assumption that $\Crm$ is second-order definable we can express ``there is $p \in \Crm$ so that\dots''. This definition works across any $(\Vrm_\kappa,\Xcal) \models \GB$ which defines $\Crm$ the same because $\Trm$ is absolute. Finally, the same argument as with $\Trm$ gives a code for the hyperclass $\Def(\Vrm_\kappa;\Crm)$.
\end{proof}

In particular, this lemma implies that ``every class is definable from $\Crm$'' is a second-order definable property. We will write $\Class = \Def(\Vrm_\kappa;\Crm)$ as an abbreviation for the second-order formula asserting this.

It remains to determine how to give an absolute definition for a generic $\Crm$. In brief, we will define $\Crm$ to be a carefully chosen Cohen generic subclass of $\kappa$, using the $\HOD$-order to ensure canonicity of any choices.

\begin{lemma}\label{le:cannonicalChoiceofC}
There is a second-order definition for $\Crm \subseteq \kappa$ which is Cohen-generic over $(\Vrm_\kappa,\Def(\Vrm_\kappa))$ so that any $\GB$ model over $\Vrm_\kappa$ defines $\Crm$ the same. Consequently, there is a second-order definition for $\Trm(\Crm)$ so that all $\GB$ models over $\Vrm_\kappa$ define $\Trm(\Crm)$ the same.
\end{lemma}

The idea behind this lemma is originally due to Feferman \cite{feferman1965}, who did the same construction in the context of arithmetic. See \cite{odifreddi1983-1,odifreddi1983-2,odifreddi1983-3} for an exposition of Feferman's work.

\begin{proof}
Recall that the forcing $\Add(\kappa,1)$ is $\mathord{<}\kappa$-closed and is first-order definable over $\Vrm_\kappa$. There are $\kappa$ many dense subsets of $\Add(\kappa,1)$ which appear in $\Def(\Vrm_\kappa)$, so we can meet them one at a time, using closure at limit stages. From $\Trm$ define a sequence $\vec D = \seq{D_\xi : \xi \in \kappa}$ of all the dense classes in $\Def(\Vrm_\kappa)$ by ordering them by the $\HOD$-least pair $(\phi,\vec a)$ which gives a dense class. Note that $\vec D$ is first-order definable from $\Trm$. Since $\Trm$ is absolutely definable this means that all models of $\GB$ over $\Vrm_\kappa$ compute $\vec D$ the same. 

The construction is done in $\kappa$ many steps. Start with $p_0 = \emptyset$. Having built $p_\xi$ define $p_{\xi+1}$ to be the $\HOD$-least condition $< p_\xi$ which meets $D_\xi$. And if $\eta$ is limit then define $p_\eta = \bigcup_{\xi < \eta} p_\xi$. Because $\kappa$ is inaccessible we have that $p_\eta \in \Vrm_\kappa$ and thus we can continue the induction. Finally, set $\Crm = \bigcup_{\xi \in \kappa} p_\xi$. Because any model of $\GB$ over $\Vrm_\kappa$ computes $\vec D$ the same, inductively we can see that they all compute each $p_\xi$ the same, whence they compute $\Crm$ the same. 
\end{proof}

We are now in a position to exhibit that $\GB$ is not semantically tight.

\begin{theorem} \label{thm:gb-not-sem-tight}
The two models $(\Vrm_\kappa,\Def(\Vrm_\kappa))$ and $(\Vrm_\kappa,\Def(\Vrm_\kappa;\Crm))$ of $\GB$, where $\kappa$ is inaccessible, $\Vrm_\kappa \models \Vrm = \HOD$, and $\Crm$ is the generic defined as above, are bi-interpretable.\footnote{Indeed, as remarked by the referee, one can alternatively establish the bi-interpretation of $(\Vrm_\kappa,\Def(\Vrm_\kappa))$ and $(\Vrm_\kappa,\Def(\Vrm_\kappa;\Crm))$ with $(\Vrm_\kappa,\Trm)$. The interpretation of $(\Vrm_\kappa,\Def(\Vrm_\kappa;\Crm))$ in $(\Vrm_\kappa,\Trm)$ is obtained by representing $\Def(\Vrm_\kappa;\Crm))$ with the $\HOD$-least codes of $\Trm$ with the additional symbol for $C$; the other direction is obtained directly by the Observation~\ref{obs:truth-definability}.}
\end{theorem}

\begin{corollary}
$\GB$ is not semantically tight.
\qed
\end{corollary}

\begin{proof}[Proof of Theorem]
Let $\Xcal = \Def(\Vrm_\kappa)$ and $\Ycal = \Def(\Vrm_\kappa;\Crm)$. Interpreting $(\Vrm_\kappa,\Xcal)$ inside $(\Vrm_\kappa,\Ycal)$ is simple. The interpretation, call it $\Ical$, is the identity on its domain, and $\in^\Ical$ is simply $\in$. The domain includes all of $\Vrm_\kappa$ to be the sets of the interpreted model, but restricts the classes to only include those which are first-order definable. This domain is second-order definable because $\Trm$ is second-order definable.

The interpretation in the other direction, call it $\Jcal$, takes more care, since we need to refer to classes which are not actually in $\Xcal$. For the sets of the interpreted model we will take all of $\{0\} \times \Vrm_\kappa$ and for the classes we will take a subset of $\{1\} \times \kappa$. Specifically, $(1,\xi)$ is in the domain of $\Jcal$ just in case $(\Trm(\Crm))_\xi \ne (\Trm(\Crm))_\eta$ for all $\eta < \xi$, where the subscripts refer to the rank of the indices in the $\HOD$-order. For sets, $(0,x) \in^\Jcal (0,y)$ if and only if $x \in y$. For set-class membership, $(0,x) \in^\Jcal (1,\eta)$ if and only if $x \in (\Trm(\Crm))_\eta$. In effect, the interpretation is that each class in the extension is interpreted as (the index of) the first formula which defines it.

It is clear from the constructions that $\Ical(\Vrm_\kappa,\Ycal) = (\Vrm_\kappa,\Xcal)$ and $\Jcal(\Vrm_\kappa,\Xcal) \cong (\Vrm_\kappa,\Ycal)$. For one composition, work inside $(\Vrm_\kappa,\Xcal)$. It is easy that $\Ical(\Jcal(\Vrm_\kappa))$ is isomorphic to $\Vrm_\kappa$---just strip off the $0$ in the first coordinate. For the classes, to define an isomorphism $\Xcal \cong \Ical(\Jcal(\Xcal))$, given a class $X$ first query $\Trm(\Crm)$ to find the $\HOD$-least formula which defines $X$. Call the index of this formula $\xi$. Then send $X$ to $(1,\xi)$. This isomorphism is first-order definable from $\Trm(\Crm)$, so it is second-order definable over $(\Vrm_\kappa,\Xcal)$, which correctly computes it. For the other composition, it is again easy that the sets of $\Jcal \comp \Ical$ are isomorphic to the sets in the ground model. For the classes, again do the same trick of looking for the $\HOD$-least slice of $\Trm(\Crm)$ which gives $X$.
\end{proof}


\subsection{Semantic non-tightness of \texorpdfstring{$\KM_k$}{KM_k}}\label{semanticNonTightnessForStrongerTheories}

Fix for the entirety of this section finite $k \ge 1$.

It will be convenient to work with the stronger theory $\KMCC_k$. This gives a slight improvement to the conclusion that $\KM_k$ is not semantically tight, so that is no cost to pay.
To show that $\KMCC_k$ is semantically non-tight we will follow the same strategy as in the previous subsection. One model of $\KMCC_k$ will be the minimum model of $\KMCC_k$ over $\kappa$ and the other will be an extension of the minimum model by a canonically chosen Cohen generic.

Fix finite $k \ge 1$. Let $\alpha > \kappa$ be the smallest ordinal so that $\Lrm_\alpha(\Vrm_\kappa)$ satisfies $\Sigma_k$-Collection and $\Sigma_k$-Separation. By the assumption that $\Vrm_\kappa \models \Vrm = \HOD$, we have a definable global well-order in $\Lrm_\alpha(\Vrm_\kappa)$, call it the \emph{$\Lrm(\Vrm_\kappa)$-order}. Set $\Mcal$ to consist of all subsets of $\Vrm_\kappa$ which appear as elements of $\Lrm_\alpha(\Vrm_\kappa)$. This $(\Vrm_\kappa,\Mcal)$ will be our minimum model of $\KMCC_k$.

\begin{lemma}\label{minModelkColSep}
The minimum model $(\Vrm_\kappa,\Mcal) \models \KMCC_k$.
\end{lemma}

\begin{proof}
It is immediate that the model satisfies Class Extensionality and Class Replacement.
Consider a $\Sigma^1_k$-formula $\phi(x)$, possibly with parameters from $\Mcal$. Inside $\Lrm_\alpha(\Vrm_\kappa)$, the set $\{ x \in \Vrm_\kappa : \Lrm_\alpha(\Vrm_\kappa) \models \phi(x)^\Mcal \}$ exists by $\Sigma_k$-Separation. But then this set is in $\Mcal$, establishing the instance of Comprehension for $\phi$. Now consider a $\Sigma^1_k$-formula $\phi(x,Y)$, possibly with parameters from $\Mcal$, and assume that for each $x \in \Vrm$ there is $Y \in \Mcal$ so that $(\Vrm_\kappa,\Mcal) \models \phi(x,Y)$. By $\Sigma_k$-Collection in $\Lrm_\alpha(\Vrm_\kappa)$ we find therein a set $b \subseteq \powerset(\Vrm_\kappa)$ so that for each $x \in \Vrm_\kappa$ there is $Y \in b$ so that $\phi(x,Y)^\Mcal$. Because $\Lrm_\alpha(\Vrm_\kappa)$ has an injection $f: b \to \kappa$ we can build the set $B = \{ (f(Y), y) : y \in Y \in b \}$, which is in $\Mcal$. This $B$ witnesses the instance of Class Collection for $\phi$, completing the proof.
\end{proof}

While this is not necessary to produce non-isomorphic but bi-interpretable models of $\KMCC_k$, we remark as an aside that this $\Mcal$ really does give a minimum model.

\begin{theorem}[Ratajcyk]
If $(\Vrm_\kappa,\Xcal) \models \KM_k$ then $\Mcal \subseteq \Xcal$.
\end{theorem}

\begin{proof}
By work of Ratajcyk \cite{ratajczyk1979}, every model of $\KM_k$ contains a submodel with the same sets which satisfies $\KMCC_k$. So we may assume that $(\Vrm_\kappa,\Xcal) \models \KMCC_k$. Let $M \models \ZFCmik{k}$ be the unrolled model, obtained as discussed in Subsection~\ref{bi-interpretability-first-order}. Because $\kappa$ has uncountable cofinality, $(\Vrm_\kappa,\Xcal)$ is correct about which of its classes are well-founded. Thus, $M$ is well-founded, and we assume without loss that $M$ is transitive. By the leastness of $\alpha$, we have $\Lrm_\alpha(\Vrm_\kappa) \subseteq M$ and thus $\Mcal \subseteq \Xcal$.
\end{proof}

Next we need to see that we can define a code for $\Mcal$ in such a way that different models of $\KMCC_k$ over $\Vrm_\kappa$  will define the same code. First, let us work with $\Lrm_\alpha(\Vrm_\kappa)$. We begin by highlighting an easy but useful fact.

\begin{lemma}
Over $\Lrm_\alpha(\Vrm_\kappa)$ there is a definable increasing cofinal map $\alpha \to \kappa$. Consequently, any outer model of $\Lrm_\alpha(\Vrm_\kappa)$ can define this map, with the same definition working uniformly across all outer models.
\end{lemma}

\begin{proof}[Proof Sketch]
The argument combines two facts. First, because $\Lrm_\alpha(\Vrm_\kappa)$ doesn't satisfy $\Sigma_{k+1}$-Replacement, there is a definable cofinal map from some $\xi < \alpha$ to $\alpha$. Second, because $\Lrm_\alpha(\Vrm_\kappa)$ satisfies that every set injects into $\kappa$ we may take $\xi = \kappa$. And it's easy to get the map to be increasing.
\end{proof}

Once we have an increasing cofinal map $f: \kappa \to \alpha$ it is straightforward to define a bijection $\kappa \to \Lrm_\alpha(\Vrm_\kappa)$. For each $f(i)$ pick the $\Lrm(\Vrm_\kappa)$-least bijection $b_i : \kappa \to \Lrm_{f(i)}$. Combining these together we get a map $\kappa \times \kappa \to \Lrm_\alpha(\Vrm_\kappa)$, and via a pairing function we may take the domain to be $\kappa$. To get a bijection we need to ensure everything in the codomain is hit only once, but this is easily done by only picking the least index. Writing down an explicit definition is tedious, but it is clear that this produces a definable map. One can think of this bijection as giving us uniform access to all of $\Lrm_\alpha(\Vrm_\kappa)$.

But we want to work over $(\Vrm_\kappa,\Mcal)$ to get a uniform access to all of $\Mcal$, which requires some small adjustments.

\begin{corollary} \label{cor:define-trm-lcal}
Over $(\Vrm_\kappa,\Mcal)$ we can define, via a second-order formula, a code $\Trm_\Mcal$ for $\Mcal$. Moreover, we can do this in such a way that any $(\Vrm_\kappa, \Ycal) \models \KMCC_k$ which is a width-extension of $\Mcal$ will define the same code $\Trm_\Mcal$.
\end{corollary}

\begin{proof}
Let $f : \kappa \to \alpha$ denote the definable, cofinal map defined above. The point is, we can mimic the definition of $f$ inside $(\Vrm_\kappa,\Mcal)$. In some detail: There is an isomorphic copy of $(\TC(\{x\}), \mathord\in \rest \TC(\{x\}))$ in $\Mcal$ for each $x \in \Lrm_\alpha(\Vrm_\kappa)$. More, by Mostowski's collapse lemma any extensional, well-founded relation with a maximum element in $\Mcal$ is isomorphic to the restriction of $\in$ to $\TC(\{x\})$ for some $x \in \Lrm_\alpha(\Vrm_\kappa)$.\footnote{Note that Mostowski's lemma is provable in $\KP + \Sigma_1$-Separation, so it holds in $\Lrm_\alpha(\Vrm_\kappa)$.}
In sum, $(\Vrm_\kappa,\Mcal)$ can mimic quantification over $\Lrm_\alpha(\Vrm_\kappa)$ by quantifying over extensional, well-founded relations with a maximum element, and thus $(\Vrm_\kappa,\Mcal)$ can mimic the definition of $f$.

We then define a code $\Trm_\Mcal \subseteq \kappa \times \kappa \times \Vrm_\kappa$ for $\Mcal$ by putting $(i,j,x)$ in $\Trm_\Mcal$ if $x$ is in the $j$-th element of $\Lrm_{f(i)}(\Vrm_\kappa)$ according to the $\Lrm(\Vrm_\kappa)$-least enumeration of $\Lrm_{f(i)}(\Vrm_\kappa)$. And this definition is absolute to width-extensions because width-extensions will define $\Lrm(\Vrm_\kappa)$ the same and so define $F$ the same.
\end{proof}

Note that this definition for the code $\Trm_\Mcal$ is not $\Sigma^1_k$ because the definition of $f$ is logically too complex. Of course we cannot hope to find a $\Sigma^1_k$ definition. For if $\Trm_\Mcal$ were $\Sigma^1_k$ definable then it would be an element of $\Mcal$ by $\Sigma^1_k$-Comprehension, but then $\Trm_\Mcal$ would be an element of $\Lrm_{\xi}(\Vrm_\kappa)$ for some $\xi < \alpha$ and so all of $\Lrm_\alpha(\Vrm_\kappa)$ would occur by stage $\xi$. That would be absurd.
\smallskip

Similar machinery works for relative constructibility. Given a class $\Crm$ over $\Vrm_\kappa$, let $\Mcal(\Crm)$ denote the subsets of $\Vrm_\kappa$ which appear in $\Lrm_\alpha(\Vrm_\kappa,\Crm)$. As in the $\GB$ case, if $\Crm$ is a generic for a forcing in $\Mcal$ then we can define a canonical choice of a code $\Trm_\Mcal(\Crm)$ for $\Mcal(\Crm)$.

\begin{lemma} \label{lem:define-trm-lcal-crm}
Suppose $\Crm \subseteq \Vrm_\kappa$ is generic over $(\Vrm_\kappa,\Mcal)$ for a the forcing $\Add(\kappa,1)$ and $\Crm$ is uniformly second-order definable in every model of $\KMCC_k$ which width-extends $\Mcal$. Then we can define, via a second-order formula, a code $\Trm_\Mcal(\Crm)$ for $\Mcal(\Crm)$in such a way that any model of $\KMCC_k$ which width-extends $\Mcal$ will define $\Trm_\Mcal(\Crm)$ the same.
\end{lemma}

\begin{proof}[Proof Sketch]
Again we use a definable cofinal map $f : \kappa \to \alpha$ to define $\Trm_\Mcal(\Crm)$. The difference  is, rather than ask about elements of levels of $\Lrm(\Vrm_\kappa)$ we ask about what conditions in $\Crm$ force. Here's one way you could implement this. Put $(i,j,x)$ in $\Trm_\Mcal(\Crm)$ if there is a condition $p \in \Crm$ which forces that $x$ is an element of the $j$-th element of $\Lrm_{f(i)}(\Vrm_\kappa,\Crm)$ according to the $\Lrm(\Vrm_\kappa,\Crm)$-least enumeration. Again the forcing lemma lets us do this definition inside $\Mcal$. This definition is uniform across width-extensions because they have the same class well-orders and thus compute $\Lrm(\Vrm_\kappa)$ the same.
\end{proof}

It remains to give the definition for $\Crm$. We use the same strategy as before to get a definition absolute for width-extensions of $(\Vrm_\kappa,\Mcal)$. From the code $\Trm_\Mcal$ we canonically extract a $\kappa$-sequence of dense subclasses of $\Add(\kappa,1)$ in $\Mcal$ and meet them one at a time. We use the $\HOD$-order in $\Vrm_\kappa$ to ensure a canonical choice at each step.

\begin{lemma}
There is a second-order definition for $\Crm \subseteq \kappa$ which is Cohen-generic over $(\Vrm_\kappa,\Mcal)$ so that any model of $\KMCC_k$ which width-extends $(\Vrm_\kappa,\Mcal)$ defines $\Crm$ the same. Consequently, there is a second-order definition for $\Trm_\Mcal(\Crm)$ so that all width extensions of $(\Vrm_\kappa,\Mcal)$ which satisfy $\KMCC_k$ will define $\Trm(\Crm)$ the same.
\qed
\end{lemma}

Recall Theorem~\ref{thm:forcing-preserves-axioms} that tame class forcing, such as adding a Cohen-generic class of ordinals, preserves $\KMCC_k$. Also recall Theorem~\ref{thm:width-extensions} that tame class forcing produces width-extensions. So $(\Vrm_\kappa,\Mcal[\Crm])$ is among the width-extensions of $(\Vrm_\kappa,\Mcal)$ subject to the conclusion of the lemma.

We are now in a position to exhibit that $\SnCA{k}$ is not semantically tight. This is analogous to the $\GB$ proof, so we omit most the details. 

\begin{theorem}
Let $k \ge 1$, let $\kappa$ be inaccessible, and let $\Mcal$ and $\Crm$ be defined as above, where we assume $\Vrm_\kappa \models \Vrm = \HOD$.
The two models $(\Vrm_\kappa,\Mcal)$ and $(\Vrm_\kappa,\Mcal(\Crm))$ of $\KMCC_k$ are bi-interpretable.
\end{theorem}

\begin{corollary}
$\KM_k$ and $\KMCC_k$ are not semantically tight.
\qed
\end{corollary}

\begin{proof}[Proof Sketch of Theorem]
Interpreting $(\Vrm_\kappa,\Mcal)$ inside the larger model $(\Vrm_\kappa,\Mcal[\Crm])$ is easy, because $\Mcal$ is a definable hyperclass in the larger model. For the other direction, use the code $\Trm_\Mcal[\Crm]$, which is second-order definable over $(\Vrm_\kappa,\Mcal)$ to give an interpretation, as in the similar direction in the proof of Theorem~\ref{thm:gb-not-sem-tight}, interpreting classes in the larger model by the ($\HOD$-least) index of their slice in $\Trm_\Mcal$.

As remarked after the lemmata, $(\Vrm_\kappa,\Mcal[\Crm])$ and $(\Vrm_\kappa,\Mcal[\Crm])$ compute $\Trm_\Mcal$ and $\Trm_\Mcal[\Crm]$ the same. This ensures that composing one interpretation with the other gives back (an isomorphic copy of) the model we started out with.
\end{proof}

\section{Non-tightness in class theory}\label{NonTightness}

To obtain the nontightness of the class theories we consider we need a uniform construction, one which applies to any model of a fixed first-order theory. We will strengthen the theories we used to ensure an appropriately modified version of the construction from Section~\ref{semanticNonTightness} goes through in a more general setting. Two key facts about $\Vrm_\kappa$ we used were its well-foundedness, ensuring uniqueness of certain constructions, and the regularity of $\kappa$, ensuring that when we constructed a Cohen generic in $\kappa$ many steps that the partial constructions were sets in $\Vrm_\kappa$. For our purposes we can replace these non-first-order axiomatizable properties with a strong form of the Replacement schema.

\begin{definition}\label{def:sor}
Let $\Phi$ be a collection of formulae in the language of set or class theory. The axiom schema of \emph{$\Phi$-Replacement} consists of the instances of Replacement for all functional $\phi \in \Phi$, i.e. the axioms
\[
\forall a\ \big( (\forall x \in a \exists y\ \phi(x,y)) \impl (\exists b \forall x \in a \exists y \in b\ \phi(x,y)) \big),
\]
allowing parameters, which we suppressed here.
Let $\SOR$ denote \emph{Second-Order Replacement}, namely $\Phi$-Replacement where $\Phi$ is the collection of all second-order formulae in the language of class theory, allowing class parameters. 
\end{definition}

It is not difficult to see that $\SOR$ is consistent, given mild large cardinals. If $\kappa$ is inaccessible then $(\Vrm_\kappa, \Xcal) \models \SOR$ for any collection $\Xcal \subseteq \powerset(\Vrm_\kappa)$ of classes over $\Vrm_\kappa$, by the regularity of $\kappa$. 

Let's collect some consequences of $\SOR$. These are proved using the same arguments for the first-order versions of the axiom/theorem schemata.

\begin{lemma}[Second-order separation]
Fix $(M,\Xcal) \models \GB + \SOR$. If $x \in M$ and $\phi(y)$ is any second-order formula, possibly with parameters, then $\{ y \in^M x : (M,\Xcal) \models \phi(y) \}$ is an element of $M$.\footnote{There is a small abuse of notation here. It could be $M$ isn't  a transitive set and $\in^M$ isn't the true $\in$. In such a case it doesn't make sense to talk about $\{ y \in^M x : (M,\Xcal) \models \phi(y) \}$ being an element of $M$. What we mean is that $M$ has an element $z$ so that $(M,\Xcal) \models z = \{ y \in x : \phi(y) \}$. We use this sort of talk rather than more precise circumlocutions because we think it clearer to stick close to how we talk about transitive models.} \qed
\end{lemma}

\begin{lemma}[Second-order recursion along $\Ord$]\label{SORec}
Fix $(M, \Xcal) \models \GB + \SOR$. Let $G \subseteq M$ be a second-order definable class function. Then there is a unique definable function $F$ over $(M, \Xcal)$ such that $F(\alpha) = G(F \rest \alpha)$ for every $\alpha \in \Ord^M$. \qed
\end{lemma}

\begin{lemma}[Second-order recursion along set-like, well-founded relations]\label{SOWFRec}
Fix $(M, \Xcal) \models \GB + \SOR$. Let $G \subseteq M$ be a second-order definable class function and $R 
\subseteq M$ be a second-order definable, set-like, well-founded relation.\footnote{We of course mean that $(M,\Xcal)$ thinks that $R$ is set-like and well-founded. In the sequel we will use similar phrasing with similar intent, and trust the reader to understand. If we wish to speak of what is seen externally to the model we will be explicit.} Then there is a unique definable function $F$ over $(M, \Xcal)$ such that $F(x) = G(F \rest \Ext_R(x))$ for every $x \in \dom R$. \qed
\end{lemma}

We highlight an immediate corollary we will make repeated use of.

\begin{corollary}
Fix $(M,\Xcal) \models \GB + \SOR$. Suppose $F \subseteq M$ is defined by second-order recursion. Then for any $x \in M$ we have that $F(x) \in M$.
\end{corollary}

\begin{proof}
Because $F(x)$ is definable by second-order Separation.
\end{proof}

\begin{lemma}[Second-order induction]
Fix $(M, \Xcal) \models \GB + \SOR$. Let $R \subseteq M$ be a second-order definable, set-like, well-founded relation. Suppose $X$ is a second-order definable, inductive subset of the domain of $R$. Then $X = \dom R$. \qed
\end{lemma}

An instance of this is especially relevant to our purposes.

\begin{corollary}
Over $\GB$, $\SOR$ proves the single sentence that asserts for every $k \in \omega$ there is a $\Sigma_k$ satisfaction predicate.
\end{corollary}

\begin{proof}
It is easy to see that the subset of $\omega$ consisting of the $k$ for which a $\Sigma_k$ satisfaction predicate exists is inductive.
\end{proof}

Just $\GB$ alone proves the existence of the $\Sigma_k$ satisfaction predicate for every standard $k$, by an induction in the metatheory. The point is, with $\SOR$ the quantification over $k$ is not in the metatheory and we get $\Sigma_k$ satisfaction predicates even for nonstandard $k$. The connoisseur of nonstandard models knows that $\omega$-nonstandard models may fail to admit any $\Sigma_k$ satisfaction predicate for nonstandard $k$.\footnote{For the non-connoisseur: Let $M$ be an $\omega$-nonstandard model of $\ZF$, and let $\Xcal$ consist of its definable classes. Then $(M,\Xcal) \models \GB$. But $\Xcal$ cannot have a $\Sigma_k$ satisfaction predicate for nonstandard $k$ by Tarski's theorem on the undefinability of truth.} 
Second-Order Replacement rules these models out from consideration.

As an aside, we remark that $\GB + \SOR$ exceeds $\GB$ in consistency strength.

\begin{proposition}
$\GB + \SOR$ proves the consistency of $\GB$.
\end{proposition}

\begin{proof}[Proof Sketch]
Work internally to a model of $\GB + \SOR$. By second-order Separation form the set of (parameter-free) first-order truths of the universe of sets. By induction this truth set contains every instance of Replacement and Separation. And it must be consistent, so we have constructed a consistent extension of $\ZFC$, whence we get the consistency of $\GB$.
\end{proof}

On the other hand, $\SOR$ says very little about what classes exist.

\begin{lemma} \label{lem:sor-goes-down}
Let $(M,\Xcal) \models \GB + \SOR$. Suppose $\Ycal \subseteq \Xcal$ is definable over $(M,\Xcal)$ by a second-order formula, possibly using parameters. Then $(M,\Ycal) \models \SOR$.
\end{lemma}

\begin{proof}
Consider an instance $\psi$ of $\SOR$. Let $\psi^\Ycal$ be the relativization of $\psi$ so that class quantifiers only quantify over elements of $\Ycal$, using that $\Ycal$ is definable. (Set quantifiers are unchanged.) By $\SOR$ we get that $(M,\Xcal) \models \psi^\Ycal$. Hence $(M,\Ycal) \models \psi$.
\end{proof}

In particular, by similar logic as to how we defined $\Def(\Vrm_\kappa)$ in the previous section, we will get that any model of $\GB$ can define what internally looks like the definable classes. This gives a model of $\GB + \SOR$ with a weak second-order theory, not even able to prove the existence of a satisfaction predicate which measures all first-order formulae. Such a model will fail to satisfy even $\PCA$.

We close this section with the fact that forcing preserves $\SOR$.

\begin{lemma} \label{lem:forcing-preserve-sor}
Suppose $(M,\Xcal) \models \GB + \SOR$. Then, any forcing extension of $(M,\Xcal)$ by a tame, $\mathord{<}\Ord$-closed forcing in $\Xcal$  will satisfy $\SOR$.
\end{lemma}

\begin{proof}
Let $(M,\Xcal[G])$ denote the forcing extension---the sets are the same by $\mathord{<}\Ord$-closure---and suppose toward a contradiction that it fails to satisfy $\SOR$. Let $\phi(x,y)$ be an instance of $\SOR$ which fails in this extension; that is, $(M,\Xcal[G])$ has a set $a$ so that for all $x \in a$ there is unique $y$ so that $\phi(x,y)$ but there is no set containing all such $y$. This is forced by some condition $p \in G$. By $\mathord{<}\Ord$-closure we may moreover assume that $p$ decides the identity of each of these witnesses; there are $\card{a} < \Ord$ many names to decide, so by closure we have enough space to continually extend to decide each of them. And they are decided to be equal to some check name $\check y$, since no sets are added. But then $(M,\Xcal)$ satisfies that there is a set $a$ so that for all $x \in a$ there is a unique set $y$ so that $p \forces \phi(\check x,\check y)$, with no set $b$ containing all such $y$. This is a failure of $\SOR$ in the ground model, contrary to the assumptions of the lemma.
\end{proof}

\subsection{Non-tightness of \texorpdfstring{$\GB$}{GB}}\label{NonTightnessGB}

All results in this section concern models of $\GB + \SOR + \Vrm = \HOD$, and the reader is warned we will not make this assumption explicit in every single definition and lemma. Many results do not need the full strength of this assumption, but we leave it to the interested reader to identify the minimal assumptions for each result.

We start this section by considering truth and definability. We take some care to make it clear everything works in the $\omega$-nonstandard case. All definitions that follow take place in the context of a fixed model of $\GB + \SOR + \Vrm = \HOD$.

\begin{definition}
A \emph{partial satisfaction predicate} is a class $S$ of (first-order) formulae $\phi[\vec a]$ equipped with set parameters assigned to all free variables so that the domain of $S$ is closed under subformulae and $S$ satisfies the Tarskian recursion on its domain. If the domain of $S$ is all $\Sigma_k$ formulae, for $k \in \omega$, we call $S$ the \emph{$\Sigma_k$ satisfaction predicate}.
\end{definition}

Our use of the definite article in that last sentence is justified by the following observation.

\begin{proposition}
Any two partial satisfaction classes agree on their common domain.
\end{proposition}

\begin{proof}
By Elementary Comprehension form the class of locations where they disagree. If nonempty there must be a minimal location $\phi[\vec a]$ of disagreement. But since they agree on subformulae of $\phi[\vec a]$ and they both satisfy the Tarskian recursion they must agree on the truth of $\phi[\vec a]$.
\end{proof}

\begin{definition}
Define $\Trm$ to be the union of all partial satisfaction predicate. Write $\Trm_k$ for the $\Sigma_k$ satisfaction predicate.
\end{definition}

It follows from earlier remarks that $\SOR$ implies $\Trm$ is the full satisfaction predicate, the unique satisfaction predicate that measures the truth of all formulae.\footnote{Uniqueness here is only inside the model of $\GB$. From the external perspective we may see multiple subsets of $M$ which satisfy the Tarskian recursion and measure the truth of all formulae in $M$. But only one of these can be a class in our model.}

A notion of satisfaction carries a notion of definability. Let $\Dcal$ be the (second-order definable) hyperclass of all $\Trm$-definable classes. That is, $X \in \Dcal$ if and only if there is $\phi[x,\vec a]$ so that $X = \{ x : \phi[x,\vec a] \in \Trm\}$. As in the $\Vrm_\kappa$ case, with minor reshuffling of indexing $\Trm$ gives a code for $\Dcal$. 

\begin{definition}
We write $\Class = \Def(V)$ to denote the axiom asserting that every class is in $\Dcal$.
\end{definition}

\begin{lemma}
Consider a model $(M,\Xcal)$ of $\GB + \SOR + \Vrm = \HOD$. Then $\Dcal^\Xcal \subseteq \Xcal$ and $(M,\Dcal^\Xcal) \models \GB + \SOR + \Vrm = \HOD$. 
\end{lemma}

\begin{proof}
For the first part, we note that $X$ is $\Trm$-definable if, and only if, $X$ is $\Sigma_k$-definable for some $k \in \omega^M$. By $\SOR$, for each $k$ the $\Sigma_k$ satisfaction predicate is in $\Xcal$. So from First-Order Comprehension we obtain $X \in \Xcal$.

For the second part:
Extensionality is trivially obtained and Replacement holds because it holds in the larger $\Xcal$. For First-Order Comprehension, fix $A \in \Dcal^\Xcal$ and assume $B$ is externally definable from $A$ via a $\Sigma_\ell$-formula with set parameters, i.e. for standard $\ell$. Because $A$ is $\Trm$-definable that means that $A$ is $\Sigma_k$-definable for some level $k$ in $\omega^M$. But then $B$ is $\Sigma_{k+\ell}$-definable, whence $B$ is $\Trm$-definable. 
Finally, that $(M,\Dcal^\Xcal)$ satisfies $\SOR$ is Lemma~\ref{lem:sor-goes-down} and that it satisfies $\Vrm = \HOD$ is because $\Vrm = \HOD$ only quantifies over sets.
\end{proof}

\begin{lemma}\label{lemma:Dabsolute}
The definition of $\Dcal$ is absolute between models with the same sets and same $\Trm$. That is, if $(M,\Xcal)$ and $(M,\Ycal)$ are models where $\Trm^\Xcal = \Trm^\Ycal$ then $\Dcal^\Xcal = \Dcal^\Ycal$.
\end{lemma}

\begin{proof}
Just observe that the definition of $\Dcal$ from the parameter $\Trm$ only quantifies over sets. 
\end{proof}

\begin{lemma}
Moving to $\Dcal$ preserves the satisfaction predicate. In symbols: $\Trm^\Dcal = \Trm$.
\end{lemma}

\begin{proof}
Because the $\Sigma_k$ satisfaction predicate is $\Sigma_{k+1}$-definable.
\end{proof}

Altogether, we have that $\Dcal$ thinks it is the minimum model of $\GB$.

\begin{corollary}
The $\Dcal$ operator is idempotent. That is, for any $(M,\Xcal) \models \GB$ we have that $\Dcal^{\Dcal^\Xcal} = \Dcal^\Xcal$.
Consequently, $(M,\Dcal^\Xcal)$ satisfies $\GB + \SOR + \Vrm = \HOD + \Class = \Def(V)$. \qed
\end{corollary}

\begin{corollary}\label{modelOfDefinableClasses}
If $(M, \Xcal) \models \GB + \SOR + \Vrm = \HOD + \Class = \Def(V)$, then $\Xcal = \Dcal^\Xcal$. \qed
\end{corollary}

These definitions and results about satisfaction/definability can be relativized to a class parameter. If this parameter is an element of $\Xcal$ then the proofs are near identical. If the parameter is a second-order definable Cohen generic then we need a slight change. As in the proof of Lemma~\ref{lem:crm-relative-truth}, the change is to ask about what is forced. We state the relativized results only for the Cohen generic case, and omit any proofs as they are the same modulo this small change.

\begin{lemma}
Fix $(M,\Xcal) \models \GB$ and suppose $\Crm \subseteq M$ is generic over $\Dcal^\Xcal$ for the forcing $\Add(\Ord,1)$. Then $\Crm \in \Dcal(\Crm)^\Xcal$ and $(M,\Dcal(\Crm)^\Xcal) \models \GB$. Moreover, if $\Crm$ is uniformly definable over models with the same sets and the same $\Trm$, then the definitions of $\Trm(\Crm)$ and $\Dcal(\Crm)$ are absolute between these models. \qed
\end{lemma}

\begin{lemma}
For any $(M,\Xcal) \models \GB$ and any Cohen-generic $\Crm$, we have that $\Dcal(\Crm)^{\Dcal(\Crm)^\Xcal} = \Dcal(\Crm)^\Xcal$. \qed
\end{lemma}

\begin{corollary}
Fix $(M,\Xcal) \models \GB$ and fix a Cohen-generic $\Crm$. Then $(M,\Dcal(\Crm)^\Xcal) \models \GB + \Class = \Def(V; \Crm)$. Note that this can be expressed as a single second-order assertion, using the parameter $\Crm$. \qed
\end{corollary}



\begin{corollary}\label{modelOfRelativeDefinableClasses}
If $(M, \Xcal) \models \GB + Class = \Def(V; \Crm)$ where $\Crm$ is a Cohen-generic over $(M,\Dcal^\Xcal)$, then $\Xcal = \Dcal(\Crm)^\Xcal$. \qed
\end{corollary}

Now we turn our attention to the definition of the Cohen generic $\Crm$ we use in our construction. 

\begin{lemma} \label{lem:cohen}
Work over $(M,\Xcal) \models \GB + \SOR + \Vrm = \HOD$. Over this model there is second-order definable $\Crm \subseteq M$ which is generic for $\Add(\Ord,1)^M$ over $(M,\Dcal^\Xcal)$. Moreover, $\Crm$ and $\Trm(\Crm)$ are absolute to any $(M,\Ycal) \models \GB + \SOR + \Vrm = \HOD$ for which $\Trm^\Ycal = \Trm^\Xcal$.
\end{lemma}

\begin{proof}
As in Lemma~\ref{le:cannonicalChoiceofC}, we define the class $D \subset \Ord \times \Add(\Ord, 1)$ such that the slices $(D)_i$ are all the $\Trm$-definable dense subclass of $\Add(\Ord, 1)$. The order in $D$ is obtained from the $\HOD$-order in $M$. We also use the $\HOD$-order to recursively build the sequence of increasingly stronger forcing condition $\seq{p_\xi : \xi \in \Ord}$ such that $p_\xi \in (D)_\xi$. This is where we use $\SOR$: This sequence is defined by transfinite recursion using a second-order definition (because we need a second-order definition to define $\Trm$ to thereby define $D$). By Lemma~\ref{SORec} this recursion succeeds and the initial segments of the sequences are sets in $M$. 

What remains in the proof is precisely the same as in the proof of Lemma~\ref{le:cannonicalChoiceofC}. We omit repeating it.
\end{proof}

A consequence of Lemma~\ref{lem:forcing-preserve-sor} is that the extension by $\Crm$ will satisfy $\SOR$. To check that the lemma we just proved includes this extension itself we simply need to check that it defines $\Trm$ the same as its ground model. Fortunately this is easy.

\begin{lemma}
Let $(M,\Xcal) \models \GB + \SOR + \Vrm = \HOD$. Then any extension of $(M,\Xcal)$ by a tame forcing in $\Xcal$ which does not add sets will define $\Trm$ the same as $\Xcal$.
\end{lemma}

\begin{proof}
Let $(M,\Xcal[G])$ denote the forcing extension. It satisfies $\GB$, so it thinks that $\Sigma_k$ satisfaction predicates are unique. Since $\Xcal[G]$ contains all of $\Xcal$, it thus agrees with $\Xcal$ as to what class is the $\Sigma_k$ satisfaction predicate for all $k$, even nonstandard. So they define $\Trm$ the same.
\end{proof}

Finally we are in a position to prove that $\GB$ is not tight.

\begin{theorem}
Consider the following two theories.
\begin{itemize}
\item $D$ is the theory consisting of $\GB + \SOR + \Vrm = \HOD + \Class = \Def(V)$.
\item $U$ is the theory consisting of $\GB + \SOR + \Vrm = \HOD + \Class = \Def(V; \Crm)$ where $\Crm$ is the Cohen generic over $\Dcal$ built up according to Lemma~\ref{lem:cohen}.
\end{itemize}
The theories $D$ and $U$ are bi-interpretable, via interpretations that fix the sets of the models.
\end{theorem}

\begin{corollary}
$\GB$ and $\GB + \SOR$ are not tight. \qed
\end{corollary}

\begin{proof}[Proof Sketch of Theorem]
We use the same interpretations $\Ical$ and $\Jcal$ from Theorem~\ref{thm:gb-not-sem-tight}. 

First we interpret $D$ in $U$ via $\Ical$, whose domain is $\Dcal$. This is expressible beacuse $\Dcal$ is a definable hyperclass. As before, $\Ical$ is the identity on its domain and $\in^\Ical$ is $\in$. The lemmata about $\Dcal$ then give that this is an interpretation of $D$ in $U$. 

For the other direction let us work in an arbitrary model $(M, \Xcal) \models D$. For the sets of the interpreted model we will take all of $\{0\} \times M$ and for the classes we will take a subset of $\{1\} \times \Ord.$
Specifically, $(1,\xi)$ is in the domain of $\Mcal$ just in case $(\Trm(\Crm))_\xi \ne (\Trm(\Crm))_\eta$ for all $\eta < \xi$, where the subscripts refer to the rank of the indices in the canonical global well-order. For sets, $(0,x) \in^\Mcal (0,y)$ if and only if $x \in y$. For set-class membership, $(0,x) \in^\Mcal (1,\eta)$ if and only if $x \in (\Trm(\Crm))_\eta$. In effect, the interpretation is that each class in the extension is interpreted as (the index of) the first formula which defines it. The lemmata about $\Crm$ imply that this is an interpretation of $U$ in $D$. In particular.

That these interpretations compose to give definable bijections is the same argument as in Theorem~\ref{thm:gb-not-sem-tight}.
\end{proof}

\subsection{Non-tightness of $\KMCC_k$}\label{semantic-non-tightness-KMk}

Our work here is to show that the construction in Section~\ref{semanticNonTightnessForStrongerTheories} can be made to work uniformly, instead of working only over a fixed transitive model. As before, it will be convenient to work with the unrolled model of $\ZFCmik{k}$ as described in Section~\ref{bi-interpretability-first-order}. We will use $\Class = \Lrm$ to mean that every class is (second-order) constructible. More precisely, $\Class = \Lrm$ expresses the translation of $\Vrm = \Lrm$ in the unrolled model. Note that $\Class = \Lrm$ implies $\Vrm = \Lrm$.\footnote{Earlier we only assumed $\Vrm = \HOD$. We think that the results in this section would go through in that more general context. But since we needed to make use of some nontrivial fine structure theory for this section we found it easier to work in this section with models where everything is constructible, rather than work with relative constructability. This is a cost, since it means our results as written cannot apply to models with large enough large cardinals. We leave it to the reader who wishes to avoid this cost to check the details for the relative constructibility context.}

Before we describe our construction, let us recall the construction in the transitive context. A transitive model of $\ZFCmik{k} + \Vrm = \Lrm$ is of the form $\Lrm_\alpha$ and thinks there is a largest cardinal, call it $\kappa$, and it is inaccessible. This $\Lrm_\alpha$ is bi-interpetable with a model of the form $(\Lrm_\kappa, \Mcal) \models \KMCC_k$. There is a least $\alpha$ which gives such a model of $\KMCC_k$ with $\Lrm_\kappa$ as the sets. It can be characterized as the smallest $\alpha > \kappa$ which satisfies $\Sigma_k$-Replacement. 

In the transitive setting, we used that such $\alpha$ must admit some $\Sigma_{k+1}$-definable cofinal map $\kappa \to \alpha$. But if we are to have a uniform construction then we must have a single definition for a cofinal map across all models, and it must be sufficiently absolute to achieve the nontightness of $\KMCC_k$. The basic idea is, we successively close off under taking witnesses for instances of $\Sigma_k$-Replacement for more and more inputs. For this we will make use of some fine structural tools.



Briefly: Jensen---e.g. in \cite{jensen1972fine}---gave a precise analysis of the structure of $\Lrm$ as built up using \textit{rudimentary functions}. He considers an alternate hierarchy, the $\Jrm$ hierarchy, to build up $\Lrm$. But the $\Jrm$ and $\Lrm$ hierarchies agree on limit levels, so the distinction will not be relevant for our purposes. A key theorem he proves is that levels of the $\Jrm$ hierarchy have $\Sigma_k$ Skolem functions for all $k$, uniformly so. Let us give a version of this appropriate to our context.


\begin{definition}\label{def-skolem-func}
Consider a model $U$ of a strong enough fragment of $\ZFC$ and fix finite $k$. We say that \emph{$U$ has a $\Sigma_k$ Skolem function} when there is a $\Sigma_k$ definition for the function $h: \omega^U \times U \to U$ such that, for every $\Sigma_k$ formula $\phi$,
\[
U \models \exists y\ \phi(y, x) \impl \phi(h(\godel{\phi}, x), x).
\]
\end{definition}


\begin{theorem}[Jensen's $\Sigma_\ell$ uniformization theorem]\label{jensenTheorem}
Fix $1 \le k \le \ell$.
There is a $\Sigma_k$ formula $\psi$ such that $\ZFCm_\ell + \Vrm = \Lrm$ proves that $\psi$ defines a $\Sigma_k$ Skolem function. That is, if $U \models \ZFCm_\ell + \Vrm = \Lrm$ then the function $h = \psi^U$ defined by $\psi$ in $U$ is a $\Sigma_k$ Skolem function for $U$.
\end{theorem}

This is a combination of some results in  \cite{jensen1972fine}, mainly Lemma 2.9. and Lemma 3.4.(i) together with the technique of standard codes developed in Section 4. 
A recent version of these can be found in the recent \cite{jensenNewBook} where Jensen develops the fine structure theory in more detail, mentioning more precisely where and how uniformization applies absolutely. 
Jensen presents his work in the context of transitive models with $\ZFC$ as the background theory. A careful read-through of his arguments makes clear that this background theory is overkill. One does not need the Powerset axiom to carry out the inductive construction, and the amount of Collection and Separation needed does not exceed the complexity of the desired Skolem function. Rather than multiply this paper's length by giving a reconstruction of Jensen's arguments with a careful accounting of the axioms used, we point the reader to the above-cited works. We also point to \cite{dodd:book} for an analysis of the minimal axioms---less than even $\KP$---needed to carry out the basic constructions of the rudimentary functions.



\begin{lemma}\label{rho-definitions}
Fix $k \ge 1$.
Work over $\KMCC_k + \SOR + \Class = \Lrm$ Consider the unrolled model $U \models \ZFCmik{k} + \Vrm = \Lrm$, and let $\kappa$ denote the largest cardinal in this model. Then there is a definition for a sequence $\seq{\alpha_i : i \in \omega}$ so that $\bigcup_i \Lrm_{\alpha_i} \models \ZFCmik{k}$, with the same largest cardinal $\kappa$. Consequently the property ``the sequence $\seq{\alpha_i}$ is cofinal in the ordinals'' is expressible.
\end{lemma}

Before the proof let's clear up a potential misunderstanding. The union $\bigcup_i \Lrm_{\alpha_i}$ refers to the direct limit of the system of models $\Lrm_{\alpha_i}$, each equipped with the membership relation from the unrolled model. If we're working over a transitive, model then this union is itself a level of the $\Lrm$ hierarchy. But in a nonstandard model there might be a cut and the sequence $\alpha_i$ doesn't have a supremum in the model. (Indeed, that is exactly what happens when the sequence is cofinal, which it will be in the models we are interested in.) After all, while the sequence is definable, its definition is too complex for the weak theory satisfied by the model to guarantee its supremum exists as an element of the model. Nonetheless, the direct system is definable and hence its direct limit is also definable.

\begin{proof}
Consider a model $(M,\Mcal) \models \KMCC_k + \SOR + \Class = \Lrm$ and work in its unrolled model $U \models \ZFCmik{k}$. From Theorem~\ref{jensenTheorem}, we have a $\Sigma_k$ Skolem function $h$ for $U$. Define the $\omega$-sequence $\seq{\alpha_i}$ as follows: Start with $\alpha_0 = \kappa$ and at successors we will pick $\alpha_{i+1}$ to be give a level of the $\Lrm$ hierarchy which is closed under $h$ for inputs from $\Lrm_{\alpha_i}$. To this purpose, define the class function
\[
W(\alpha) = \{ \xi \in \Ord : \xi = h(\godel{\phi}, x) \text{ where } x \in \Lrm_{\alpha} \mand \phi \text{ is } \Sigma_k\}.
\]
The function $h$ is $\Sigma_k$ and being $\Lrm_{\alpha}$ is a $\Sigma_1$ property of $\alpha$. So in all $W$ is $\Sigma_k$, and so $W(\alpha)$ is a set in $U$. Then, set
\[
\alpha_{i+1} = \bigcup W(\alpha_i).
\]
If $\psi(y,x)$ is $\Sigma_k$ then the property ``$\Lrm_\xi$ contains $h(\godel{\psi},x)$'' is also $\Sigma_k$ in parameters $\xi$ and $x$. So this definition really does give that $\Lrm_{\alpha_{i+1}}$ is closed under $h$ for inputs from $\Lrm_{\alpha_i}$.

We can always continue the construction one more step, and so the set of $i$ for which $\alpha_i$ is defined forms an inductive subset of $\omega$. So by $\SOR$ it must be all of $\omega$.

To show that $N = \bigcup_i \Lrm_{\alpha_i} \models \ZFCmik{k}$ we first show that any $\Sigma_k$ formula $\theta(y,x)$ with a parameter $x$ in one of the $\Lrm_{\alpha_i}$ reflects. To this end fix such $\theta(y,x)$ and $x \in \Lrm_{\alpha_i}$. Suppose that $U \models \exists y\ \theta(y,x)$. Then $U \models \theta(h(\godel\theta,x))$. But $h(\godel\theta,x) \in \Lrm_{\alpha_{i+1}}$. By Tarski--Vaught we get that $N$ is a $\Sigma_k$ elementary submodel of $U$. Now by a standard argument we get that $N \models \Sigma_k$-Replacement. Namely, suppose there is $a \in N$ so that for each $x \in a$ there's a unique $y$ so that $\phi(x,y)$, where $\phi$ is a $\Sigma_k$ formula, possibly with parameters. Then there's $i$ so that $a$ and all parameters are in $\Lrm_{\alpha_i}$. Using that $N$ is $\Sigma_k$-elementary in $U$, we get that the witnesses $y$ must all be in $\Lrm_{\alpha_{i+1}}$, witnessing that instance of Replacement. This immediately implies that $N \models \Sigma_k$-Separation and Collection, where for the second fact we use that $N$ has a definable global well-order. That $N \models \Vrm = \Lrm$ is immediate. And $N$ has the same $\kappa$ as its largest cardinal by an inductive argument. Trivially $\Lrm_{\alpha_0} = \Lrm_\kappa$ has cardinality $\kappa$. And then inductively $\Lrm_{\alpha_{i+1}}$ is a union of $\kappa$ many sets of size $\kappa$ whence it's also of cardinality $\kappa$.
\end{proof}

This theorem allows us a characterization of models of class theory which think they are the minimum model of $\KMCC_k$. Namely, let $\Class = \MinMod_k$ be the conjunction of $\Class = \Lrm$ and ``in the unrolling the sequence $\seq{\alpha_i}$ is cofinal in the ordinals''. Moreover, we can express whether a model is a width-extension of a model of $\Class = \MinMod_k$, by expressing that the second-order constructible classes of that model satisfy $\Class = \MinMod_k$.

As in the transitive case, having a definable cofinal sequence in the ordinals of the unrolled model allows us to define a code for the hyperclass of all classes.

\begin{lemma}
Fix $k \ge 1$. Let $(M,\Mcal) \models \KMCC_k + \SOR + \MinMod_k$. Then over $(M,\Mcal)$ there is a definition for a code $\Trm_\Mcal$ for $\Mcal$. Moreover, this definition can be chosen to be absolute across all width-extensions of $(M,\Mcal)$.
\end{lemma}

\begin{proof}[Proof Sketch]
As in Corollary~\ref{cor:define-trm-lcal}, but since our cofinal sequence has length $\omega$ we define the code to consist of triples $(i,j,x)$ where $i \in \omega$, $j \in \Ord$, and $x$ is a set. To make the definition absolute across width-extensions, relativize it to second-order $\Lrm$.
\end{proof}

We can also do this for Cohen generics over a model of $\MinMod_k$, as in Lemma~\ref{lem:define-trm-lcal-crm}.

\begin{lemma}
Fix $k \ge 1$. Let $(M,\Mcal) \models \KMCC_k + \SOR + \MinMod_k$ and suppose $\Crm \subseteq M$ is a generic over $(M,\Mcal)$ for $\Add(\Ord,1)$ which is second-order definable over $(M,\Mcal)$. Then there is a definition for a code $\Trm_\Mcal(\Crm)$ which is absolute across all width-extensions of $(M,\Mcal)$. \qed
\end{lemma}

The extension by $\Crm$ is a width-extension, so it is among those extensions for which the definition of $\Trm_\Mcal(\Crm)$ is absolute.

Finally, we must say how to define $\Crm$. But there is no new content here. Work in a model of $\KMCC_k + \SOR + \MinMod_k$. Using the code $\Trm_\Mcal$ we extract a canonical $\Ord$-sequence of the dense subclasses of $\Add(\Ord,1)$ in the model. We extend to meet these subclasses one at a time, always using the $\Lrm$-order to make choices of how to extend. Here $\SOR$ comes into play to ensure this construction never takes us outside the model. So in $\Ord$ many steps we produce $\Crm$. And any extension which defines $\Trm_\Mcal$ the same will define $\Crm$ the same.

Let $\Class = \MinMod_k[\Crm]$ be the second-order formula which expresses that the classes are precisely those which appear in the code $\Trm_\Mcal[\Crm]$. Intuitive, this formula expresses that the model is the forcing extension of the minimum model of $\KMCC_k$ by the canonical choice of a Cohen generic.

Following the same interpretation strategy as before, we get that any model of $\Class = \MinMod_k$ is bi-interpretable with its extension by $\Crm$.

\begin{theorem}
Fix $k \ge 1$. The following two theories are bi-interpertable.
\begin{enumerate}
    \item $D_k = \KMCC_k + \SOR + \Class = \MinMod_k$.
    \item $U_k = \KMCC_k + \SOR + \Class = \MinMod_k[\Crm]$. \qed
\end{enumerate}
\end{theorem}

\begin{corollary}
Fix $k \ge 1$.
The theories $\KMCC_k$ and $\KMCC_k + \SOR$ are not tight. \qed
\end{corollary}

\section{Non-tightness in second-order arithmetic}\label{NonTightnessArithmetic}

The constructions used in the previous section also work in the context of second-order arithmetic. However, there are enough subtleties and notational differences in arithmetic context that for ease of exposition we discuss it separately in this section. Most proofs carry over \textit{mutatis mutandis} from the class theory context, and we leave it to the interested reader to rewrite the proofs with the changed details.
We state some facts about models of second-order arithmetic without proof, and we point the reader to Simpson's book on the subject \cite{simpson:book}, especially Chapter VII, for proofs and detailed references. 

In the class theory context, to get the failure of tightness we added the full second-order Replacement schema to our theories. In the arithmetic context, the analogue is the full Induction schema, i.e. the instances of Induction for every second-order formula, and we will include it in our theories to ensure constructions go through the model's full $\omega$.

The strategy is the same as in class theory. For fragments of second-order arithmetic we can write down a theory which characterizes a minimum model. We can define a canonical code for this minimum model, and thereby define a canonical Cohen extension of the minimum model. These two models are bi-interpretable. Of course, with the nonstandardness phenomenon there is no hope for an notion of minimum absolute between all models of the theory. But we can get a sufficiently absolute notion to enable the bi-interpretation, so we get the fragment of $\Zsf_2$ cannot be tight.

First we discuss the analogue of $\GB$. The theory $\ACA_0$ has as its principle axioms Induction and Comprehension for arithmetical formulae. If you strengthen Induction to the full schema, over all second-order formulae, you get the theory $\ACA$. Every $\omega$-model of $\ACA_0$---viz. a model whose numbers are isomorphic to $\omega$---automatically satisfies full $\ACA$. But for nonstandard models the theories diverge. For instance, analogous to the situation with $\GB$ and $\SOR$, over $\ACA$ you can prove that the $\Sigma_k$-satisfaction class exists for all $k$, even nonstandard. Whereas with just $\ACA_0$ you are only guaranteed to have such for standard $k$. The reason, of course, is that $(M,\Def(M))$ is always a model of $\ACA_0$ for any $M \models \PA$, and no nonstandard $\Sigma_k$-satisfaction class can be definable.

Following the $\GB$ context, we can write down a theory which identifies the minimum $\omega$-model, namely the arithmetical reals, among all $\omega$-models. Using full Induction, this will allow a definition sufficiently absolute among nonstandard models to enable two distinct but bi-interpretable extensions of $\ACA$.

There is a second-order definition for the (first-order) satisfaction predicate, call it $\Trm$. Indeed, the same definition as before---viz. the union of the $\Sigma_k$-satisfaction classes---will do, modulo the details of G\"odel coding. Note that full Induction is used to ensure there is a $\Sigma_k$-satisfaction predicate for every $k$ in the model. This makes possible a theory $D$ expressing $\ACA$ $+$ ``every set is arithmetical'' and a theory $U$ expressing $\ACA$ $+$ ``there is a canonical Cohen-generic $\Crm$ over the arithmetical sets and every set is arithmetical in $\Crm$''. These two theories are then bi-interpretable, with two key points---proved much the same as the class theoretic case---being that forcing preserves full and induction and the definition of $\Trm$.

All in all, we get the following result.

\begin{theorem}
The theory $\ACA$ is not tight. Consequently, any weakening of $\ACA$ in the language of second-order arithmetic, such as $\ACA_0$, is also not tight. \qed
\end{theorem}

For stronger fragments of $\Zsf_2$ the same basic strategy works, but producing the code for the minimum model is more difficult than defining a satisfaction predicate. We start by recalling some definitions and facts.

\begin{definition} \label{def:soas}
Fix $k \ge 1$.
\begin{itemize}
\item The theory $\PnCA{k}_0$ is obtained from $\ACA_0$ by adding Comprehension for $\Pi^1_k$ formulae.
\item The theory $\PnCA{k}$ is obtained from $\PnCA{k}_0$ by adding full Induction.
\item The theory $\SnAC{k}_0$ is obtained from $\PnCA{k}$ by adding the $\Sigma^1_k$-Choice schema. This schema is the arithmetic counterpart to the $\Sigma^1_k$-Class Collection schema; cf. Definition~\ref{def:cc}.
\item The theory $\SnAC{k}$ is obtained from $\SnAC{k}_0$ by adding full Induction.
\end{itemize}
\end{definition}

The theory $\ATR_0$, a strict subtheory of $\PCA_0$, is strong enough to carry out the unrolling construction. As such, theories of arithmetic which extend $\ATR_0$ are bi-interpretable with certain set theories. These theories are strong enough to carry out the construction of $\Lrm$. Restricting to the constructible sets gives fragments of the AC schema, so $\SnAC{k}$ does not exceed $\PnCA{k}$ in consistency strength. Unlike in class theory, however, weak enough fragments of the $\AC$ schema are outright provable, without any assumption of every set being constructible.

\begin{theorem}
The $\Sigma^1_1$-Choice schema is a consequence of $\ATR_0$, and over $\ATR_0$ the $\Sigma^1_2$-Choice schema is equivalent to $\Delta^1_2$-Comprehension. For $k > 2$, the $\Sigma^1_k$-Choice schema is a consequence of $\PnCA{k}_0$ $+$ ``there is a real from which every real is constructible''.
\end{theorem}

For our purposes we are looking at models which satisfy that every real is constructible, or Cohen-extensions thereof. So we will only be looking at models of $\SnAC{k}$. (This will include $\SnAC{1} = \PCA$ and $\SnAC{2} = \PnCA{2}$, but for the sake of uniform notation we will use the former names.) Here are bi-interpretation results for these theories.

\begin{theorem}
The following pairs of theories are bi-interpretable.
\begin{itemize}
\item $\Zsf_2$ $+$ $\Sigma^1_\infty$-$\mathsf{CA}$ and $\ZFCm$ plus ``every set is countable''.
\item For $k \ge 1$, $\PnCA{k+1}_0$ and $\ZFCm_{k}$ plus ``every set is countable''.
\end{itemize}
\end{theorem}

In the class theoretic case, the indexing was the same. Here they are off by one. The culprit is well-foundedness. In class theory this is a first-order property, whereas in arithmetic it is $\Pi^1_1$-universal. Because of this misaligned indexing, the situation with $\SnAC{1}$ is different from the stronger theories. We discuss it first.

\begin{theorem}
The minimum $\beta$-model of $\SnAC{1}$ consists of the reals in $\Lrm_{\omega_\omega^\mathrm{CK}}$, where $\omega_\omega^\mathrm{CK}$ is the supremum of the first $\omega$ many admissible ordinals.
\end{theorem}

Note that the inclusion of full Induction means that any model of $\SnAC{1}$ thinks the $n$-th admissible ordinal $\omega_n^\mathrm{CK}$ exists even for nonstandard $n$. This is because $\SnAC{1}_0$ is strong enough to prove that the admissible ordinals are unbounded and so the set of such $n$ is inductive. And over $\SnAC{1}$ we can define the sequence of the $\omega_n^\mathrm{CK}$. From this sequence we can extract a canonical code of all the reals in $\Lrm_{\omega_\omega^\mathrm{CK}}$, as in the similar argument for strong fragments of $\KM$. From this code we can define, via a second-order formula, a canonical choice for a Cohen real $\Crm$ which is generic over $\Lrm_{\omega_\omega^\mathrm{CK}}$. 

It is straightforward to formulate an axiom asserting over $\SnAC{1}$ that every set is in $\Lrm_{\omega_\omega^\mathrm{CK}}$. Namely, this axiom asserts that for every set $X$ there is an integer $n$ so that there is a length $n$ sequence of well-orders $\gamma_i$ so that each $\gamma_i$ is admissible, $\gamma_0 = \omega$, there are no admissibles between $\gamma_i$ and $\gamma_{i+1}$, and $X$ is in $\Lrm_{\gamma_n}$. Call this axiom $\Class = \Adm_\omega$. Similarly, we can formulate an axiom $\Class = \Adm_\omega[\Crm]$ which asserts that $\Crm$ exists and every set is in $\Lrm_{\omega_\omega^\mathrm{CK}}[\Crm]$, where $\Crm$ is definable Cohen generic over the reals in $\Lrm_{\omega_\omega^\mathrm{CK}}$. We follow the same strategy as before to define $\Crm$, using the canonical code of $\Lrm_{\omega_\omega^\mathrm{CK}}$.

Putting this all together, we get that the theories $\SnAC{1} + \Class = \Adm_\omega$ and $\SnAC{1} + \Class = \Adm_\omega[\Crm]$ are bi-interpretable. A key point is, a model of $\SnAC{1}$ and any Cohen-extension thereof will have the same well-orders and agree on which well-orders give admissible ordinals. So they will define the canonical code for the reals in $\Lrm_{\omega_\omega^\mathrm{CK}}$ the same, and thereby define $\Crm$ the same.

\begin{theorem}
The theory $\SnAC{1}$ ($= \PCA$) is not tight. Consequently, any weakening of $\SnAC{1}$ in the language of second-order arithmetic, such as $\PCA_0$, is also not tight. \qed
\end{theorem}

We turn at last to $\SnAC{k}$ for $k \ge 2$. For these the characterization of the least $\beta$-model is more complex.

\begin{theorem}
For $k > 1$, the minimum $\beta$-model of $\SnAC{k}$ consists of the reals in $\Lrm_\alpha$ where $\alpha$ is the least ordinal so that $\Lrm_\alpha$ satisfies $\Pi_{k-1}$-Separation. Equivalently, $\alpha$ is the least ordinal so that $\Lrm_\alpha$ satisfies $\Sigma_{k-1}$-Replacement. Equivalently, $\alpha$ is the least ordinal whose $k$-th projectum $\rho^\alpha_k$ is itself.
\end{theorem}

Yet again, the sticking point is defining a cofinal $\omega$-sequence in $\alpha$. Here we can use the same fine structural facts as in Section~\ref{semantic-non-tightness-KMk}. If $(M,\Xcal) \models \SnAC{k}$ $+$ ``every set is constructible'' then it is bi-interpretable with its unrolling $U \models \ZFCm_{k-1} + \Vrm = \Lrm$. By Theorem~\ref{jensenTheorem} $U$ has a $\Sigma_n$ Skolem function. Using this Skolem function and full Induction we can define an $\omega$-sequence of ordinals $\alpha_n$ which are cofinal in what $U$ thinks is the minimum model of $\ZFCm_{k-1}$, as in the proof of Lemma~\ref{rho-definitions}. We can thus formulate an axiom expressing that $U$ is itself this minimum model, namely by asserting that every set is in $\Lrm_{\alpha_n}$ for some $n$. 

All this can be translated over to the model $(M,\Xcal)$ of second-order arithmetic. Let $\Class = \MinMod_k$ denote the formula in the language of second-order arithmetic asserting that every set is in this minimum model. Then any model of $\SnAC{k} + \Class = \MinMod_k$ can define a canonical code for itself, using the definable $\omega$-sequence $\seq{\alpha_n}$ of ordinals cofinal in the height of its unrolling. This code will be absolute between this model and its Cohen-extensions, because they will have the same well-orders and hence the same $\Lrm$. From this code we can define a canonical choice for a Cohen real $\Crm$ which is generic over the minimum model. And so we can formulate an axiom $\Class = \MinMod_k[\Crm]$ which asserts over $\SnAC{k}$ that $\Crm$ exists and every real is in the extension of the minumum model by $\Crm$.

Following the same interpretations as used in the previous theorems, we then conclude that $\SnAC{k} + \Class = \MinMod_k$ and $\SnAC{k} + \Class = \MinMod_k[\Crm]$ are bi-interpretable.

\begin{theorem}
Fix $k \ge 2$. Then, $\SnAC{k}$ is not tight. Consequently, any theory weaker than $\SnAC{k}$ in the language of arithmetic, such as $\PnCA{k}$ and $\PnCA{k}_0$, are also not tight. \qed
\end{theorem}


\section{Final remarks} \label{sec:final}

The reader who thoroughly read the previous sections will have noticed that this article is about essentially one construction done over and over in different settings. We remark that it may be carried out in yet more settings. For example, \cite{williams_2019} proves that there is a minimum $\beta$-model of $\ETR$, where \emph{Elementary Transfinite Recursion} $\ETR$ is the class theoretic analogue of $\ATR_0$. One can take their construction of the minimum $\beta$-model of $\ETR$ and put the construction of this article in that setting, thereby showing that the minimum $\beta$-model of $\ETR$ is bi-interpretable with its extension by a canonical choice of Cohen generic.

In \cite{enayat2017variations}, Enayat conjectures that no (proper) subsystem of the tight theories $\PA$, $\Zsf_2$, $\ZF$, $\KM$ can be tight. In this article, we demonstrated the non-tightness of $\KM_k + \SOR$ for all natural numbers $k$. 
These results yield a natural and comprehensive collection of non-tight subtheories approximating $\KM$. 
Indeed, it shows that full comprehension is the minimal level of comprehension that produces a tight theory from $\GB$.

As we were writing this article, we learned from Ali Enayat that in forthcoming work \cite{enayat:in-progress} he had independently proved that all finitely axiomatizable fragments of $\PA$, $\Zsf_2$, $\ZF$, $\KM$ are not tight. His proof provides an alternate proof of the nontightness of $\GB$ and $\KM_k$, along with their arithmetical counterparts. But his argument does not apply to $\GB + \SOR$, $\KM_k + \SOR$ and their arithmetical counterparts, as the second-order Replacement schema (respectively, the second-order Induction schema) is not finitely axiomatizable.
Thus, putting together Enayat's results, our results in this aritcle, and those of Freire and Hamkins \cite{bwst2020}, we have a substantial basis for Enayat's conjecture. And while there still are other theories to consider in order to assert that all subtheories of KM are not tight, these would be quite unnatural subtheories.

What is missing to get the full result? With respect to class theories, we should consider proper subtheories of $\KM$ that have instances of Comprehension of ever growing formula complexity, but in such a way that full $\KM$ isn't provable from those instances. This type of subsystem is hardly considered in the literature and for this reason it may require some novel treatment.
The status of same kind of the fragment of arithmetic and first-order set theory is also unknown. 
Additionally, while one may consider levels of formula complexity in the single scheme (induction) in arithmetic and (comprehension) in class theory, the same do not apply to $\ZF$ as one should also consider fragments of Replacement together with the full scheme of Separation, or other natural systems like Zermelo set theory plus the assertion that every set is in a $\Vrm_\alpha$. Therefore, while the picture may be said to be nearly complete for $\KM$, $\PA$ and $\Zsf_2$, there is still a lot to be discovered with respect to $\ZF$.

\printbibliography

\end{document}